\documentclass{amsart}
\usepackage[english]{babel}
\usepackage{amsrefs,amssymb,amsmath,amsthm,amsfonts,epsfig,graphicx,psfrag}
\usepackage[dvips]{color}
\usepackage{hyperref}
\hypersetup{
    colorlinks,
    citecolor=black,
    filecolor=black,
    linkcolor=black,
    urlcolor=black
}

\theoremstyle{plain}
\newtheorem{thm}{Theorem}[section]
\newtheorem{cor}[thm]{Corollary}
\newtheorem{lem}[thm]{Lemma}
\newtheorem{prop}[thm]{Proposition}

\theoremstyle{definition}
\newtheorem{df}[thm]{Definition}

\theoremstyle{remark}
\newtheorem{rmk}[thm]{Remark}
\newtheorem{ejp}[thm]{Example}
\newtheorem{question}[thm]{Question}

\DeclareMathOperator{\dive}{div}
\DeclareMathOperator{\diam}{diam}
\DeclareMathOperator{\clos}{clos}

\DeclareMathOperator{\dist}{dist}

\newcommand{\fix}{\hbox{Sing}}
\newcommand{\sing}{\fix}

\newcommand{\T}{\mathbb T}
\newcommand{\R}{\mathbb R}
\newcommand{\Q}{\mathbb Q}

\newcommand{\Z}{\mathbb Z}
\newcommand{\N}{\mathbb N}
\newcommand{\X}{\mathbb X}
\newcommand{\reparam}{\mathcal{H}_0^+(\R)}
\renewcommand{\epsilon}{\varepsilon}

\newcommand{\nUparrow}{\Uparrow\hspace{-.34cm}\diagup\hspace{.31cm}}
\begin{document}

\author[A. Artigue]{Alfonso Artigue}
\email{artigue@unorte.edu.uy}
\address{DMEL, Regional Norte, Universidad de la Rep\'ublica, Gral. Rivera 1350, Salto, Uruguay.}
\title{Kinematic Expansive flows}
\date{\today}
\begin{abstract}
In this paper we study kinematic expansive flows on compact metric spaces, surfaces and general manifolds. 
Different variations of the definition are considered and its relationship 
with expansiveness in the sense of Bowen-Walters and Komuro is analyzed. 
We consider continuous and smooth flows and robust kinematic expansiveness of vector fields is considered on smooth manifolds.
\end{abstract}
\maketitle

% \tableofs
\section{Introduction}
Let us start explaining the meaning of \emph{kinematic expansiveness} discussing a well known physical example. 
Consider the differential equation of a simple pendulum:
$ \ddot \theta + \sin(\theta)=0.$
It is known since Galileo Galilei that the period of the oscillations is almost constant if the amplitude is small. 
But, if $T(\theta_0)$ is the period of an oscillation of amplitude $\theta_0$
it can be proved that $T$ is strictly increasing for $\theta_0\in [0,\pi)$ (see for example \cite{BB} for a proof).
Consider two close initial positions of the pendulum with vanishing initial velocities. 
Since the periods of the oscillations are different, we have that the solutions will be separated at some time. 
This is the meaning of kinematic expansiveness. 

A key point for a pendulum clock as a practical timekeeper is that this separation time is large. 
In fact, it is easy to see that the separation is linear in time. 
This is a special feature of kinematic expansiveness, 
they are not so chaotic as a system with exponential error propagation.

In this paper we mainly study dynamical systems on surfaces, therefore we consider the usual change of variables 
$x=\theta$ and $y=\dot\theta$, to transform the equation of the pendulum into a first order differential equation in the plane. 
Consider two periodic solutions $\gamma_1$ and $\gamma_2$ bounding an annulus $A$ in the plane. 
If $\phi\colon \R\times A\to A$ is the action of $\R$ on $A$ induced by the pendulum equations we have our first example 
of a kinematic expansive flow on a compact surface. 
Precise definitions are given in the following section. 

The above considerations are related with the stability (or unstability) of trajectories in the sense of Lyapunov. 
In dynamical systems, another fundamental concept is the \emph{structural stability} due to Andronov and Pontryagin. 
A system is structurally stable if there is a neighborhood of the system (in a specified topology) 
such that every system of this neighborhood has an \emph{equivalent} behavior. 
In the discrete time case, two diffeomorphisms $f,g$ of a manifold $M$ are equivalent or \emph{conjugated} 
if there is a homeomorphism $h\colon M\to M$ such that $f\circ h= h\circ g$. 
In the continuous time case, one can say that two flows $\phi$ and $\psi$ are conjugated if there is 
a homeomorphism $h$ as before such that $\phi_t\circ h=h\circ \psi_t$ for all $t\in\R$. 
This concept is very restrictive, because if there is a closed 
trajectory then its period should be preserved under perturbations. 
But this is impossible because slightly changing the velocities of the system one obtains a small perturbation (on any reasonable topology) 
and the periods of the perturbed system are different. 
Therefore, we must consider the concept of \emph{topological equivalence}. 
Two flows are topologically equivalent if there is a homeomorphism that preserves trajectories and orientations. 
If the homeomorphism is the identity of the phase space we have that each flow is a (global) \emph{time change} of the other. 
It is also called a \emph{reparametrization} of the flow. 

If one is allowed to change the velocities of single trajectories, the distance between two whole orbits can be measured 
with the \emph{Fr\'echet distance}. If $\alpha,\beta\colon \R\to X$ are continuous curves on a metric space $(X,\dist)$, 
then the Fr\'echet or \emph{geometric} distance between the curves is
\[
 \dist_F(\alpha,\beta)=\inf_{h}\sup_{t\in \R} \dist(\alpha(t),\beta(t)),
\]
where $h\colon\R\to\R$ varies in the set of increasing homeomorphisms of $\R$.
This takes us to the concept of \emph{geometric expansiveness}, 
that is similar to kinematic expansiveness but allowing time reparametrizations of trajectories. 
This concept was first considered in the literature by Anosov 
to prove the structural stability of now called Anosov flows. 

Later, Bowen and Walters introduced a definition of expansive flow, see \cite{BW}, 
that on arbitrary compact metric spaces allowed them to prove some properties shared with Anosov flows. 
Their definition is of a \emph{geometric} nature, that is, they require that trajectories are separated 
even allowing time changes of single orbits.
In \cite{BW} it is noticed that kinematic expansiveness is not enough in order to recover results of hyperbolic flows. 
In the introduction of cited paper they consider a flow topologically equivalent with the pendulum system described above. 
Some of the results in \cite{BW} were generalized 
in \cite{KS79} considering different families of reparametrizations and acting groups.

A different and very interesting kind of expansiveness was discovered by Gura. 
In \cite{Gura}, he proved that the horocycle flow of a surface with negative curvature is positive and negative 
(kinematic) separating, his definition requires to separate every pair of points in different orbits. 
He also proved a remarkable result: every \emph{global} time change of such flow is positive and negative kinematic separating. 
It is known that horocycle flow is not geometric expansive.

The aim of this paper is to study kinematic expansiveness. 
Examples and basic properties are mainly stated on compact surfaces. 
A special feature of kinematic expansiveness is the non-invariance under global time changes. 
Therefore we also consider the definition of strong kinematic expansiveness requiring that every global time change 
must be kinematic expansive. 
A natural question is: why call it strong kinematic and not weak geometric? 
The answer can be found in the following example. 
Consider $X$ a vector field in a two-dimensional torus $\T^2$ generating an irrational flow. 
Take a non-negative map $\rho$ with just one zero at $p\in\T^2$. 
Define the vector field $Y=\rho X$ and let $\phi$ be its associated flow. 
As we will see, $\phi$ and its global time changes are kinematic expansive. 
The separation of trajectories is not geometric because generic orbits are parallel straight lines.

This paper is organized as follows.

Section \ref{hierarchy}. We define and state the basic properties of expansive and separating flows 
in the kinematic, strong kinematic and geometric versions. 
Examples are given to analyze the relationship between the definitions. 

Section \ref{HieOnSurf}. We consider flows on compact surfaces. 
We prove that on surfaces every geometric separating flow is geometric expansive (i.e. $k^*$-expansive in the sense of Komuro \cite{K}).
We also show that a flow on a surface is strong kinematic expansive if and only if 
it is strong separating; and also equivalent with:
its singularities are of saddle type and 
the union of the separatrices is dense in the surface. 

Section \ref{secSuspension}. We study the kinematic expansiveness of suspension flows. 
We found a dynamical characterization of the topology of compact subsets of the real line 
related with kinematic expansive suspensions.
We give a characterization of arc homeomorphisms admitting a kinematic expansive suspension. 
We prove that the only $C^1$ diffeomorphism of an interval admitting a kinematic expansive suspension 
is the identity. A similar study is done for circle homeomorphisms and diffeomorphisms.

Section \ref{secKinSurf}. We consider kinematic expansive flows of surfaces. 
We study the relationship between singularities and kinematic expansiveness in the disc and in the annulus. 
We show that every compact surface admits a kinematic expansive flow. 

Section \ref{secPositive}. For positive geometric expansive flows on compact metric spaces it is known (see \cite{Artigue2}) 
that the dynamic is trivial (a finite number of compact orbits). 
The case of positive kinematic expansiveness is different, as is shown in the examples of this section. 
We study the local behavior of the flow near a compact orbit. 
On surfaces, we prove that positive expansive flows are suspensions and has no singularities. 
The smooth case is also considered. 
We consider a variation of an example in \cite{KS} to show (on a compact metric space) that 
a positive kinematic expansive flow may not be negative kinematic expansive. 

% Section \ref{secBiExp}. In \cite{Gura} it is shown that the horocycle flow of a surface of negative curvature is 
% positive and negative separating. It is known that this is a minimal flow. 
% We prove that a flow on compact manifold of dimension greater than 2 with a compact orbit cannot be 
% kinematic bi-expansive (i.e., in the positive and negative senses).

Section \ref{secRobust}. In this section we consider perturbations of $C^1$ kinematic expansive vector fields. 
We first give a sufficient condition for robust kinematic expansiveness for conservative vector fields in the annulus. 
Finally we consider a general manifold and vector fields without singularities. 
In this setting we prove that $C^1$-robust kinematic expansiveness implies geometric expansiveness (i.e., expansiveness 
in the sense of Bowen and Walters \cite{BW}).

\section{Hierarchy of expansive flows}
\label{hierarchy}
In this section we present the main definitions.
Let $(X,\dist)$ be a compact metric space and $\phi\colon \R\times
X\to X$ be a continuous flow. 
We say that $p\in X$ is a \emph{singularity} or an \emph{equilibrium} point of $\phi$ if $\phi_t(p)=p$ for all $t\in\R$.
To understand any definition of expansive flow one must consider the following 
simple fact.
\begin{rmk}
It holds that if there is at least one non-singular point $x\in X$ then for all 
$\delta>0$ there exists $s\in\R$, such that
$y=\phi_s(x)\neq x$ and for all $t\in\R$, $\dist(\phi_t(x),\phi_t(y))<\delta$. 
Moreover, the value of $s$ may be as small as we want.
\end{rmk}
Then, if we are going to define
a notion of expansive or separating flow we must take care of points in the same orbit. 
In the subject of expansive flows we consider the hierarchy shown in Table \ref{tablaExp}.
\begin{table}[h]
\[
\begin{array}{ccc}
\hbox{\fbox{Geometric Expansive}} & \Rightarrow & \hbox{\fbox{Geometric Separating}} \\
\Downarrow && \Downarrow\\
\hbox{\fbox{Strong Kinematic Expansive}} & \Rightarrow & \hbox{\fbox{Strong Separating}}\\
\Downarrow && \Downarrow\\
\hbox{\fbox{Kinematic expansive}} & \Rightarrow & \hbox{\fbox{Separating}}\\
\end{array}
\]
\caption{Hierarchy of expansive flows}
\label{tablaExp}
\end{table}
The terms \emph{kinematic} and \emph{geometric} first appear in the literature of expansive systems (to our best knowledge) in \cite{CL} (page 138).
Definitions in the left column of the table separate every pair of points not being in the same \emph{local} orbit,
and the ones in the right separate points in different \emph{global} orbits. 
Strong expansiveness deals with time changes of the whole flow and the geometric notions
allows time changes of single orbits. 
As the reader will see, the implications indicated by the arrows are easy to prove.
Now we give the precise definitions, examples and 
counterexamples showing that no arrow in the table can be reversed in the general setting of compact metric spaces. 
We also state some basic properties. 

\subsection{Kinematic expansive flows} 

Let us start with the main notion of the paper. 

\begin{df} 
\label{dfkexp}
We say that $\phi$ is \emph{kinematic expansive} if 
for all $\epsilon>0$ there
exists $\delta>0$ such that if $\dist(\phi_t(x),\phi_t(y))<\delta$
for all $t\in\R$ then there exists $s\in\R$ such that
$|s|<\epsilon$ and $y=\phi_s(x)$.
\end{df}

In \cite{KS79} kinematic expansiveness is considered with the name $\{\hbox{id}\}$-\emph{expansiveness}. 
This means that the only reparametrization allowed is the identity of $\R$.
This definition was also mentioned in (the first section of) \cite{BW}. 

Two continuous flows $\phi\colon \R\times X\to X$ and $\psi\colon \R\times Y\to Y$ are said to be \emph{equivalent} 
if there exists a homeomorphism $h\colon X\to Y$ such that $\phi_t = h^{-1}\circ \psi_t\circ h$ for all $t\in \R$.

\begin{rmk}
Clearly, kinematic expansiveness is invariant
under flow equivalence, i.e., it does not depend on the metric defining the topology of $X$. 
\end{rmk}

A continuous flow $\psi\colon \R\times X\to X$ is \emph{topologically equivalent} with $\phi\colon\R\times Y\to
Y$ if there is a homeomorphism $h\colon X\to Y$ such that for each $x\in X$ the orbits $\phi_\R(h(x))$ and $\psi_\R(x)$  and its orientations coincide. 
If in addition, the homeomorphism $h$ is the identity of $X$, we say that $\phi$ is a \emph{time change} of $\psi$.

The following example is topologically equivalent with the pendulum system (restricted to an annulus as mentioned above) and 
shows that a time change can
destroy kinematic expansiveness.

\begin{ejp}[Periodic band]\label{annulus}
Consider the annulus in the plane 
\[
 A=\{(x,y)\in\R^2:x^2+y^2\in[1,4]\}
\]
bounded by circles of radius 1 and 2. 
A flow $\phi$ on $A$ can be defined by the equation
\[
(\dot x,\dot y)=\frac 1{\sqrt{x^2+y^2}}(-y,x). 
\]
The solutions are circles as shown in Figure \ref{perAnnulus}.
It is easy to see that this flow is kinematic expansive (there is a proof in Example 2 of \cite{KS79} page 84). 
\begin{figure}[htbp]
\begin{center}
\includegraphics{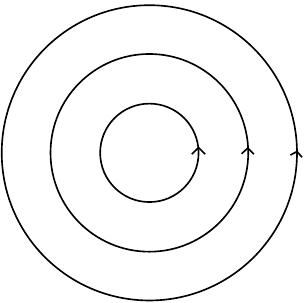}
    \caption{Periodic orbits in the annulus.}
   \label{perAnnulus}
\end{center}
\end{figure}
The equation $(\dot x,\dot y)=(-y,x)$ defines a time change of $\phi$ that is not kinematic expansive because the angular velocities are constant.
\end{ejp}

\begin{rmk}
 At the end of Definition \ref{dfkexp} above, we required that the points $x$ and $y$ are in a orbit segment of small time. 
 Since our phase space is compact, this segment has a small diameter too. 
 Notice that the converse is true if there are no singularities, i.e., orbit segments of small diameter 
 are defined by small times. 
 But if there is a singularity accumulated by regular orbits this is no longer true as the reader can verify, 
 just apply the continuity of the flow at the singular point.
\end{rmk}

In spite of this remark we will show that if we require in Definition \ref{dfkexp} that $x$ and $y$ are in a 
orbit segment of small diameter (instead of small time) we obtain an equivalent definition.
Let us first introduce another distance in $X$ by
$$\dist_\phi(x,y)=\inf\{\diam(\phi_{[a,b]}(z)):z\in X, [a,b]\subset\R, x,y\in\phi_{[a,b]}(z)\}$$
if $x,y$ are in the same orbit and $\dist_\phi(x,y)=\infty$ in other case. 
Of course, this metric will define a different topology on $X$ if $X$ is not just a periodic orbit.

\begin{prop}
A flow $\phi$ is kinematic expansive if and only if 
for all
$\beta>0$ there exists $\delta >0$ such that if
$\dist(\phi_t(x),\phi_t(y))<\delta$ for all $t\in\R$ then
$\dist_\phi(x,y)<\beta$.

\end{prop}

\begin{proof}
($\Rightarrow$) Consider $\beta>0$ and take $\epsilon>0$
such that $\dist(\phi_t(x),x)\leq\beta$ for all $x\in X$ and $t\in
[-\epsilon,\epsilon]$. By hypothesis, there exists $\delta$ such
that if $\dist(\phi_t(x),\phi_t(y))<\delta$ for all $t\in\R$ then
there is $s\in \R$ such that $|s|<\epsilon$ and $\phi_s(x)=y$.
Then $\dist_\phi(x,y)<\beta$ and the proof ends.

($\Leftarrow$) First we fix $\epsilon>0$.
Take any $\beta_1>0$ and by hypothesis there exists $\delta_1$ such that if
$\dist(\phi_t(x),\phi_t(y))<\delta_1$ for all $t\in\R$ then $\dist_\phi(x,y)<\beta_1$. 
It is easy to see that there is
just a finite number of orbits with diameter smaller than
$\delta_1/2$.
Now take $\beta_2>0$ such that if $\diam (\phi_\R(p))<\beta_2$
then $p$ is a singular point. For this value of $\beta_2$ there is
an expansive constant $\delta_2$ (by hypothesis).

Take $\rho<\delta_2/2$ and denote by $\sing$ the set of singular points of the flow. It is easy to see that for all $x\in
B_\rho(\fix)=\cup_{q\in\sing}B_\rho(q)$, $x\notin\fix$, there exists $t_0\in\R$ such that
$\phi_{t_0}(x)\notin B_\rho(\fix)$. We will prove that there is
$\beta_3\in (0,\beta_2)$ such that if $x\notin B_{\rho}(\fix)$ and
$\diam(\phi_{[0,t]}(x))<\beta_3$ then $|t|<\epsilon$. By contradiction suppose there exists $x_n\to z$,
$x_n\notin B_\rho(\fix)$ and $t_n\to\infty$ such that
$\diam(\phi_{[0,t_n]}(x_n))\to 0$. This implies that $z\in\fix$
which is a contradiction. 

Finally we claim that $\delta_3$ is an expansive constant associated to
$\epsilon$. In order to prove it, suppose that
$\dist(\phi_t(x),\phi_t(y))<\delta_3$ for all $t\in\R$. We can
assume that $x$ is not a singular point and then there exists
$t_0\in\R$ such that $\phi_{t_0}(x)\notin B_\rho(\fix)$. Then
$$\dist(\phi_t(\phi_{t_0}(x)),\phi_t(\phi_{t_0}(y)))<\delta_3$$ and
the hypothesis implies that there is $s\in\R$ such that
$\phi_s(\phi_{t_0}(x))=\phi_{t_0}(y)$ and also the diameter of
$\phi_{[0,s]}(\phi_{t_0}(x))$ is smaller than $\beta_3$. Since
$\phi_{t_0}\notin B_\rho(\fix)$ we have that $|s|<\epsilon$ and
then $\phi_s(x)=y$ with $|s|<\epsilon$.
\end{proof}

\subsection{Strong kinematic expansive flows}
As we saw in Example \ref{annulus} kinematic expansiveness is not an invariant property 
under time changes of flows. Therefore the following definition is natural.

\begin{df}
A flow is said to be \emph{strong kinematic expansive} if every time change is kinematic expansive. 
\end{df}

\begin{ejp}\label{torussing}
Consider an irrational flow (every orbit is dense) on the two dimensional torus $T^2=\R^2/\Z^2$ with velocity field $X$. 
Take any non-negative smooth function $f$ with just one zero at some point $p$ in the torus. 
Denote by $\phi$ the flow 
generated by the vector field $fX$. 
Such flow is illustrated in Figure \ref{IrrFlow}. 
To show that $\phi$ is strong kinematic expansive consider any time change of the flow.
The idea is the following. Take two points being in different local orbits and wait until one of them is very close to $p$.
By continuity this point will stay close to $p$ for a long time while the other point will be separated. This argument will
be formalized latter in Theorem \ref{skexpsurf}.
\begin{figure}[htbp]
\begin{center}
\includegraphics{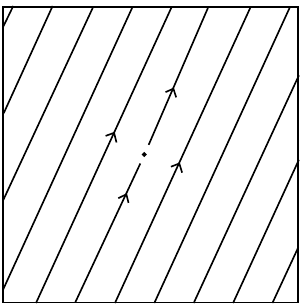}
    \caption{Irrational flow in the torus with a fake saddle.}
   \label{IrrFlow}
\end{center}
\end{figure}

\end{ejp}

\begin{rmk}
The periodic flow in the annulus shown in Example \ref{annulus} is kinematic expansive but not in the strong sense. 
\end{rmk}

\subsection{Geometric expansive flows} 
The idea of geometric expansiveness is that the trajectories separate even if one allows a time change of the trajectories.
Denote by $\reparam$ the set of all increasing homeomorphisms $h\colon \R\to\R$ such that $h(0)=0$.
Such homeomorphisms will be called \emph{time reparametrizations}.
\begin{df} We say that $\phi$ is \emph{geometric expansive}  if for all $\beta>0$ there
exists $\delta>0$ such that if
$\dist(\phi_{h(t)}(x),\phi_t(y))<\delta$ for all $t\in\R$ with
$h\in\reparam$ then $x,y$ are in a orbit segment of diameter smaller than $\beta$.
\end{df}

In the literature these flows are simply called \emph{expansive}.
In the case of regular flows, i.e. without equilibrium points, 
it is equivalent with the one given by R. Bowen and P. Walters in
\cite{BW}. 
For the general case (i.e. with or without singular points) the definition is equivalent with the given
by M. Komuro in \cite{K} (see \cite{Artigue} for a proof). 
Examples of geometric expansive flows are suspensions of expansive homeomorphisms \cite{BW}, Anosov flows, 
the Lorenz attractor \cite{K} and singular suspensions of expansive interval exchange maps \cite{Artigue}.

\begin{rmk}
It is easy to see that Examples \ref{annulus} and \ref{torussing} are not geometric expansive.
\end{rmk}

\subsection{Separating flows} The term \emph{separating} was first used in \cites{Gura,Gura75}. 
This kind of expansiveness only separates points in different global orbits.

\begin{df}A flow $\phi$ is \emph{separating} if there is $\delta>0$ such that if 
$\dist(\phi_t(x),\phi_t(y))<\delta$ for all $t\in\R$ then $y\in\phi_\R(x)$.
\end{df}

\begin{ejp}[Minimal separating flow in the torus] 
\label{ejpDH}
In \cite{DH} it is defined a continuous (non-smooth) time change of an irrational 
flow on the two dimensional torus $T^2$ with the following property:
the set $\{(\phi_t(x),\phi_t(y)):t\geq 0\}$ is dense in 
% the future limit set of $(x,y)$ is 
$T^2\times T^2$ whenever $x$ and $y$ are on different orbits in $T^2$.
Clearly, it implies that the flow is separating.
We were not able to decide if this example is kinematic expansive or not.
\end{ejp} 

The following is an easy example showing that there are separating flows that are not kinematic expansive.

\begin{ejp}[A separating flow in the M\"oebius Band]
\label{exmoebius}
 Consider the map $f\colon [-1,1]\to[-1,1]$ given by $f(x)=-x$ and consider $T(x)=1+x^2$ for all $x\in[-1,1]$. 
 It is easy to see that the suspension flow of $f$ with return time $T$ is a separating flow in the M\"oebius band. 
 See Figure \ref{figMoebius}.
\begin{figure}[htbp]
\begin{center}
\includegraphics{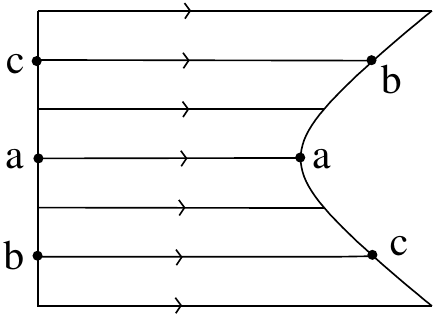}
    \caption{A separating flow in the Moebius band.}
   \label{figMoebius}
\end{center}
\end{figure}
 
\end{ejp}

\begin{rmk}
 The previous example in the Moebius band is not kinematic expansive. 
 Consider two points as $c$ and $b$ in Figure \ref{figMoebius}. 
 They are in the same orbit but not in a small orbit segment. 
 Taking them close to the point $a$ we contradict kinematic expansiveness.
\end{rmk}

Let us give some general remarks that hold for every notion of expansiveness considered in this article.
Recall that the definition of \emph{separating flow} is the weaker in Table \ref{tablaExp}.

\begin{df}
 A singularity $p$ of $\phi$ is $\phi$-\emph{isolated} if there is 
 $\delta>0$ such that for all $x\in B_\delta(p)$, $x\neq p$, there is $t\in\R$ such that 
 $\phi_t(x)\notin B_\delta(p)$.
\end{df}

\begin{rmk}
\label{rmkSing}
 If $\phi$ is separating and $p\in X$ is a singular point of the flow then 
 $p$ is $\phi$-isolated. 
 In particular, the set of singular points is finite. 
\end{rmk}

\subsection{Strong separating flows} 

\begin{df} 
A flow is \emph{strong separating} if every time change is separating. 
\end{df}

The following is a remarkable example.

\begin{ejp}
\label{horociclico}
In \cite{Gura} it is shown that the horocycle flow on a surface of negative curvature is strong separating. 
In fact, Gura shows that the separation of trajectories occurs in positive and in negative times. 
\end{ejp}

\begin{rmk} 
The example shown in \cite{DH} (recall Example \ref{ejpDH}) is separating but it is not strong separating. 
\end{rmk}

\subsection{Geometric separating flows} 
\begin{df} A flow $\phi$ is said to be \emph{geometric separating} 
if there exists $\delta>0$ such that if 
$\dist(\phi_{h(t)}(x),\phi_t(y))<\delta$ for all $t\in\R$ and some
$h\in\reparam$ then $y\in\phi_\R(x)$.
\end{df}

To study geometric separating flows we introduce a natural definition for dynamics with discrete time. 
\begin{df}
 \label{dfSepHomeo}
We say that a homeomorphism $f\colon Y\to Y$ is 
\emph{separating}\footnote{In \cite{Gura75} Gura calls \emph{separating} to what we call \emph{expansive homeomorphism}. We use the expression \emph{separating homeomorphism} with a different meaning.} if there is $\delta>0$ such that 
if $\dist(f^n(x),f^n(y))< \delta$ for all $n\in\Z$ then there is $m\in\Z$ such that $y=f^m(x)$.
\end{df}

\begin{prop}
 A suspension is geometric separating if and only if the suspended homeomorphism is separating.
\end{prop}

\begin{proof}
 Is similar to the proof of Theorem 6 in \cite{BW}.
\end{proof}

Now we give an example showing that separating homeomorphisms may not be expansive.

\begin{ejp}
\label{sepnoexp}
Let $X$ be the subset of the sphere $\R^2\cup\{\infty\}$ given by
$$
  X=\{\infty\}\cup\{(n,0):n\in\Z\}\cup\{(n,\pm1/m): n\in\Z, m\in \Z^+,|n|\leq m\}.
$$ 
Define $f\colon X\to X$ as
$f(\infty)=\infty$, $f(n,0)=(n+1,0)$, $f(n,\pm1/m)=(n+1,\pm1/m)$
if $n<m$ and $f(m,\pm1/m)=(-m,\mp1/m)$. It is easy to see that $f$ is a homeomorphism. 
It is not expansive because the points $(0,1/m)$ and $(0,-1/m)$ contradicts expansiveness
for arbitrary small expansive constants. 
It is a separating homeomorphism because these are the only points contradicting expansiveness and they are in the same orbit.
Therefore, the suspension of this example is not geometric expansive but it is geometric separating. 
\end{ejp}

\begin{rmk}
We also have that the suspension flow in the previous example 
is strong separating but it is not strong kinematic expansive.
\end{rmk}

\subsection{Summary of counterexamples}

We have defined six variations of expansive and separating flows on compact metric spaces.
In the following table we recall the counterexamples in the hierarchy:

\begin{table}[h]
\[
\begin{array}{ccc}
\hbox{\fbox{Geometric Expansive}}        & \begin{array}{c}\nLeftarrow\\ \hbox{Example }\ref{sepnoexp} \end{array} & \hbox{\fbox{Geometric Separating}} \\
\nUparrow \hbox{Examples }\ref{torussing}\hbox{ and } \ref{horociclico} && \nUparrow \hbox{Examples }\ref{torussing}\hbox{ and }\ref{horociclico}\\
\hbox{\fbox{Strong Kinematic Expansive}} & \begin{array}{c}\nLeftarrow\\ \hbox{Example }\ref{sepnoexp} \end{array} & \hbox{\fbox{Strong Separating}}\\
\nUparrow \hbox{Example }\ref{annulus} && \nUparrow \hbox{Example }\ref{ejpDH}\\
\hbox{\fbox{Kinematic expansive}}        & \begin{array}{c}\nLeftarrow\\ \hbox{Example }\ref{exmoebius} \end{array} & \hbox{\fbox{Separating}}\\
\end{array}
\]
\caption{Diagram of counterexamples}
\label{HieGen}
\end{table}
As we can see in the diagram, the six definitions are different in 
the general context of continuous flows on compact metric spaces.

\section{Hierarchy of expansiveness on surfaces}
\label{HieOnSurf}
In this section we will show that the hierarchy of expansive flows presented in Table \ref{tablaExp} is simpler, 
see Table \ref{HieOnSurf},
if we assume that the phase space is a compact surface. 
\begin{table}[h]
\[
\begin{array}{c}
\fbox{Geometric Expansive 
$\Leftrightarrow$  
Geometric Separating} \\
\Downarrow \\
\fbox{Strong Kinematic Expansive 
$\Leftrightarrow$  
Strong Separating}\\
\Downarrow \\
\fbox{Kinematic expansive}\\
\Downarrow \\
\fbox{Separating}\\
\end{array}
\]
\caption{Hierarchy of expansive flows of compact surfaces}
\label{HierOnSurf}
\end{table}
The first equivalence in Table \ref{HierOnSurf} is given in Theorem \ref{teogsepsup} and the second one is proved in Theorem \ref{skexpsurf}.
To prove these results we will first study the local behavior near singular points and time changes 
of flows with wandering points.

\subsection{Isolated singular points}

In this section we study the local behavior of singularities of separating flows of surfaces. 
Let $\phi\colon \R\times S\to S$ be a continuous flow on a compact surface $S$.
As mentioned in Remark \ref{rmkSing}, every singular point is $\phi$-isolated if $\phi$ is separating.

 Let us introduce some definitions. 
A regular orbit $\gamma$ is a \emph{separatrix} of $p\in\sing$ if for $x\in\gamma$ it holds that 
$\lim_{t\to+\infty}\phi_t(x)=p$ (\emph{unstable separatrix}) or $\lim_{t\to-\infty}\phi_t(x)=p$ (\emph{stable separatrix}).
A singular point is said to be a 
\emph{(multiple) saddle} if it presents a finite number of separatrices. We say that $p\in\sing$ is a 
$n$-\emph{saddle} if $p$ is a multiple saddle of index $n$ (i.e. if it has $n$-1 stable separatrices).

Recall that a singular point $p$ is (Lyapunov) \emph{stable} if for all $\epsilon>0$ there is $\delta>0$ such that if $\dist(x,p)<\delta$ then $\dist(\phi_t(x),p)<\epsilon$ for all $t\geq 0$. 
We say that $p$ is \emph{asymptotically stable} if it is stable and there is $\delta_0>0$ such that 
if $\dist(x,p)<\delta_0$ then $\phi_t(x)\to p$ as $t\to+\infty$. 
If $p$ is asymptotically stable we say that $p$ is a \emph{sink}. 
We say that $p$ is a \emph{source} if it is a sink for $\phi^{-1}$ defined as $\phi^{-1}_t=\phi_{-t}$.

Let us recall from \cite{Hartman} some well known facts and notations relative to the Poincar\'e-Bendixon Theory. 
Let $p\in \sing$ be a $\phi$-isolated singular point. 
Consider a Jordan curve $C$ bounding a neighborhood $U$ of $p$ such that 
if $\phi_\R(x)\subset \clos U$ then $x=p$. 
If for some $y\in C$ it holds that $\phi_{\R^+}(y)\subset U$ then we say that 
$\phi_{\R^+}(y)$ is a \emph{stable separatrix arc} (or a \emph{base solution} in the terminology of \cite{Hartman}). 
Since $p$ is $\phi$-isolated, we have that $\lim_{t\to\infty}\phi_t(y)=p$.
In the same conditions, if $\phi_{\R^-}(y)\subset U$ then this orbit segment is called \emph{unstable separatrix arc}. 

Suppose that $y_1,y_2\in C$ determine two separatrix arcs. An open subset $S$ bounded by $p$, 
the separatrix arcs of $y_1$ and $y_2$, 
and an arc in $C$ from $y_1$ to $y_2$ is called a \emph{sector}.
Notice that each pair of separatrix arcs determines two sectors. 

A sector $\sigma$ is \emph{hyperbolic} if contains no separatrix arc. 
A sector $\sigma$ determined by two stable (or two unstable) separatrix arcs is called \emph{parabolic} 
if it contains no unstable (or stable) separatrix arc. 
With reference to \cite{Hartman}, elliptic sectors needs not to be consider because $p$ is $\phi$-isolated. 
The number of hyperbolic sectors is finite by Lemma 8.2 in \cite{Hartman}. 
\begin{prop}
\label{propSectores}
Assume that $U$ is an isolating neighborhood of $p\in\sing$ bounded y a Jordan curve $C$ as above.
If the closures of all the hyperbolic sectors 
are deleted from $U$ then the residual set is either: 
\begin{enumerate}
 \item empty and $p$ is a multiple saddle,
 \item $U$ and $p$ is a sink or a source or 
 \item the union of a finite number of pairwise disjoint parabolic sectors.
\end{enumerate}
\end{prop}

\begin{proof}
See Lemma 8.3 of \cite{Hartman}.
\end{proof}

In Figure \ref{figSectores} the three possible cases of Proposition \ref{propSectores} are illustrated.

\begin{figure}[htbp]
\begin{center}
\includegraphics{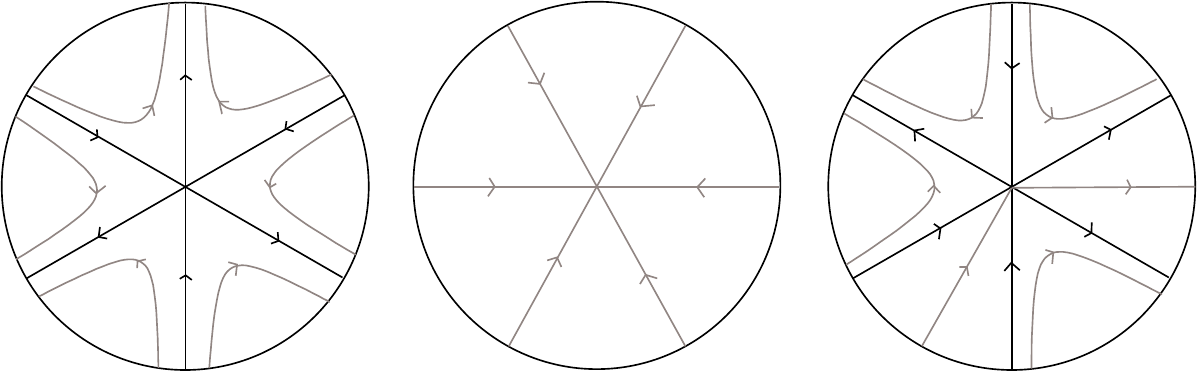}
    \caption{Examples of isolated singularities. 
    Left: a multiple saddle. 
    Center: a sink. 
    Right: a combination of hyperbolic and parabolic sectors.}
   \label{figSectores}
\end{center}
\end{figure}

\begin{df}
 Let $R$ be an embedded disc in $S$ and define a rectangle $K=[-1,1]\times[0,1]\subset\R^2$. We say that:
 \begin{enumerate}
  \item $R$ is a \emph{regular flow box} if $\phi$ restricted to $R$ is topologically equivalent with the constant vector field 
  $X(x,y)=(1,0)$ restricted to $K$, 
  \item $R$ is a \emph{parabolic flow box} if $\phi$ restricted to $R$ is topologically equivalent with 
  $X(x,y)=\pm(x,y)$ restricted to $K$, 
  \item $R$ is a \emph{hyperbolic flow box} if $\phi$ restricted to $R$ is topologically equivalent with 
  $X(x,y)=(x^2+y^2,0)$ restricted to $K$.
 \end{enumerate}
 In the last two cases we say that $R$ is a \emph{singular flow box}.
\end{df}

\begin{prop}\label{cubrimiento} If a flow $\phi$ on a compact surfaces $S$ 
presents a finite number of isolated singularities
then $S =\cup_{i=1}^{n}R_i$ where:
\begin{itemize}
\item each $R_i$ is a regular or singular flow box and
\item if $i\neq j$ then $R_i\cap R_j\subset\partial R_i\cap \partial R_j$.
\end{itemize}
\end{prop}

\begin{proof}
 It follows by Proposition 4.3 of \cite{Gutierrez} and Proposition \ref{propSectores} above.
\end{proof}

\subsection{Time changes and wandering points}

Let $\phi$ be a continuous flow on a compact surface $S$.

\begin{thm}
\label{teoSepWan}
 If $\phi$ is a continuous flow on a compact surface $S$ 
 and $\phi$ has 
 wandering points then 
 there is a time change of $\phi$ that is not separating. 
\end{thm}

\begin{proof}
If $\phi$ has a non-isolated singular point then $\phi$ is not separating. Therefore, we will assume that 
all the singularities are isolated.

 Let $l\subset S$ be a compact arc transversal to the flow such that $\phi_t(l)\cap l=\emptyset$ for all $t\neq 0$. 
 Let $L=\phi_\R(l)$. 
 Consider a covering of boxes $R_1,\dots,R_n$, $S=\cup_{i=1}^nR_i$, as in Proposition \ref{cubrimiento}. 
 Divide $l$ with two interior points in three sub-arcs $l=l_1\cup l_2\cup l_3$ in such a way that 
 $\phi_\R(l_2)$ intersects each $\partial R_i$ only at the transversal part.
 We will show that there is a time change $\psi$ of $\phi$ such that 
 for all $\delta>0$ there are $x,y\in l_2$, $x\neq y$, such that $\dist(\psi_t(x),\psi_t(y))<\delta$ 
 for all $t\in\R$.
 
 Fix a box $R_i$ such that $\phi_\R(l_2)\cap R_i\neq\emptyset$. Assume first that $R_i$ is a regular flow box. 
 The boundary of $R_i$ is the union of two transversal arcs $a$ and $b$ and two orbit segments. 
 Suppose that the flow enters to the box through $a$. 
 Given two points $x,y\in a$ the sub-arc of $a$ with extreme points $x,y$ will be denoted by $[x,y]$. 
 Call $x_1$ and $x_2$ the extreme points of $a$ as shown in Figure \ref{figCajaRegular}. 
\begin{figure}[htbp]
\begin{center}
\includegraphics{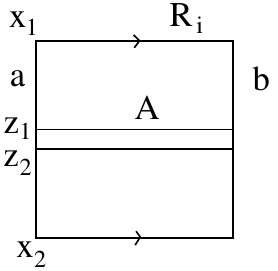}
    \caption{}
   \label{figCajaRegular}
\end{center}
\end{figure}

 Take
 $z_1,z_2\in a$ such that 
 $[z_1,z_2]\cap \phi_\R(l_2)=\emptyset$. 
Since $R_i$ is a regular flow box there is a homeomorphism $h\colon R_i \to K=[-1,1]\times [0,1]$ taking orbit segments in $R_i$ into 
horizontal segments in $K$. 
For each $p\in R_i$ denote by $\gamma(p)$ the preimage by $h$ of the vertical segment through $h(p)$. 
Each $\gamma(p)$ is a compact arc transversal to the flow.
Consider a time change $\psi$ 
such that:
\begin{enumerate}
 \item if $x\in [x_1,z_1]$ 
then $\psi_t(x)\in \gamma(\phi_t(x_1))$ for all $t\in[0,T_1]$ where $\phi_{T_1}(x_1)\in b$ 
and $\phi_{[0,T_1]}(x_1)\subset R_i$, 
\item if $x\in [z_2,x_2]$ 
then $\psi_t(x)\in \gamma(\phi_t(x_2))$ for all $t\in[0,T_2]$ where $\phi_{T_2}(x_2)\in b$ 
and $\phi_{[0,T_2]}(x_2)\subset R_i$.  
\end{enumerate}

Now consider a hyperbolic box $R_i$. 
Again denote by $a=[x_1,x_2]\subset \partial R_i$ the transversal part of the boundary of $R_i$ where the flow enters to the box. 
Consider $u_i,v_i\in a$, for $i\in \Z^+$, such that $u_1<v_1<u_2<v_2<\dots$ and 
$[u_i,v_i]\cap \phi_\R(l_2)=\emptyset$ for all $i=1,2,3,\dots$. 
Denote by $p\in \partial R_i$ the singular point in the boundary of $R_i$. 
Again, with a homeomorphism $h\colon R_i\to K$ we have a transversal (vertical) foliation on $R_i\setminus\{p\}$. 
Inside $R_i$ consider three flow boxes $A_1,B_1,C_1$ bounded by orbit segments and vertical arcs as in Figure \ref{cajasDivididas}. 
\begin{figure}[htbp]
\begin{center}
\includegraphics{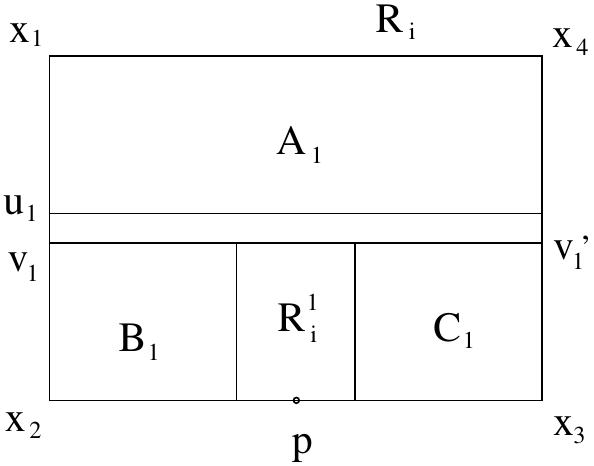}
    \caption{}
   \label{cajasDivididas}
\end{center}
\end{figure}
Also consider the hyperbolic flow box $R_i^1\subset R_i$ as in the figure. 
Define 
\[
 \begin{array}{l}
T_a=\sup\{t>0:\phi_{[0,t]}(x_1)\subset R_i\},\\
T_b=\sup\{t>0:\phi_{[0,t]}(x_2)\subset B_1\},\\
T_c=\sup\{-t>0:\phi_{[t,0]}(x_3)\subset C_1\},
 \end{array}
\]
where $x_3$ is the vertex of the box $R_i$ shown in Figure \ref{cajasDivididas}.
Consider the time change $\psi$ satisfying: 
\begin{enumerate}
 \item if $x\in [x_1,u_1]$ 
then $\psi_t(x)\in \gamma(\phi_t(x_1))$ for all $t\in[0,T_a]$, 
\item if $x\in [v_1,x_2]$ 
then $\psi_t(x)\in \gamma(\phi_t(x_2))$ for all $t\in[0,T_b]$,  
\item if $x\in [v_1',x_3]$ 
then $\psi_{-t}(x)\in \gamma(\phi_{-t}(x_3))$ for all $t\in [0,T_c]$.
\end{enumerate}

Inside $R_i^1$ consider a similar subdivision considering the orbit segments of $u_2,v_2$ as in Figure \ref{cajasDivididas2}.
\begin{figure}[htbp]
\begin{center}
\includegraphics{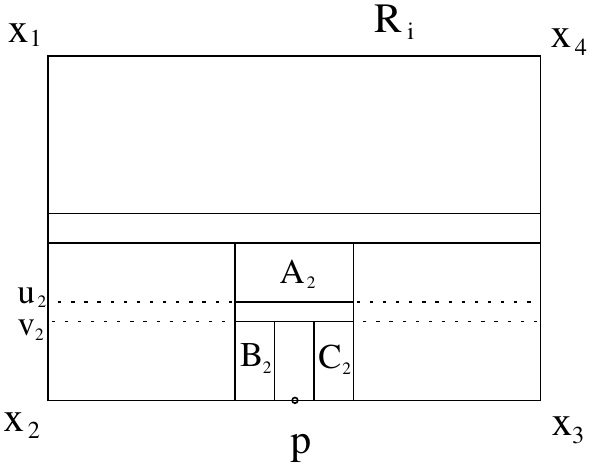}
    \caption{}
   \label{cajasDivididas2}
\end{center}
\end{figure}
Inductively we have a sequence of regular boxes $A_k,B_k,R_k$ and hyperbolic boxes $R^i_k$. 
On each $R^i_k$ assume that $\psi$ satisfies the corresponding conditions as in $R^i_1$.
Assume that $\diam(R_i^k)\to 0$ as $k\to\infty$.

On parabolic boxes, assume that $\psi$ coincides with $\phi$.

In this way we obtain a (global) flow $\psi$ that is a time change of $\phi$ and $\psi$ 
is not separating because on each box $R_i\cap\phi_\R(l_2)$ the flow $\psi$ preserves the vertical foliation of the box.
\end{proof}

\subsection{Geometric separating and geometric expansive flows on surfaces}
\label{secGsepSurf}

Let us recall that in \cite{Artigue} (see Theorem 6.7) it is proved that 
a flow on a compact surface $S$ that is not a torus, is geometric expansive
\footnote{Notice that in cited paper \emph{expansive} means \emph{geometric expansive} in the present terminology.}
if and only if 
the set of singular points is finite
and there are neither 
wandering points nor periodic orbits.
We do not consider singular points as periodic orbits.

\begin{lem}
\label{lemaPer}
 If $\phi$ is a strong separating flow on a compact surface then $\phi$ has no periodic orbits.
\end{lem}

\begin{proof}
By Theorem \ref{teoSepWan} we have that there are no wandering points. 
Therefore, if $\gamma$ is a periodic orbit, every point close to $\gamma$ has to be periodic 
(this can be easily proved by considering a local cross section through $\gamma$ and its first return map). 
But if a periodic orbit is accumulated by periodic orbits then 
there is a time change of $\phi$ that is not strong separating. 
Therefore, a strong separating flow cannot have periodic orbits.
\end{proof}

\begin{lem}
 \label{toroNoG-exp}
 The torus does not admit geometric separating flows.
\end{lem}

\begin{proof}
Assume by contradiction that $\phi$ is a geometric separating flow on the torus.
We know by Theorem \ref{teoSepWan} and Lemma \ref{lemaPer} that $\phi$ has neither wandering points 
nor periodic orbits. 
Since $\phi$ is separating, we have that the singular points are $\phi$-isolated. 
Applying Proposition \ref{propSectores} and the fact that there are no wandering points 
we have that every singular point is of saddle type, that is because there are neither sources, sinks nor parabolic sectors. 
Since the Euler characteristic of the torus equals zero we have that singular points are 0-saddles (sometimes called \emph{fake saddles}). 
Consider another flow $\psi$ that removes the singularities of $\phi$, i.e., satisfying: 1) $\psi$ has no singular points and 2) 
every orbit of $\phi$ is contained in a orbit of $\psi$.
It is known that under these conditions (see for example Lemma 4.1 in \cite{Artigue}) $\psi$ is an irrational flow, i.e. a suspension of an irrational rotation of the circle.
But now it is easy to see that $\phi$ cannot be geometric separating. 
This contradiction proves the lemma.
\end{proof}

\begin{thm}
\label{teogsepsup}
 A continuous flow on a compact surface is geometric separating if and only if it is geometric expansive.
\end{thm}

\begin{proof}
We only have to prove the direct part because the converse holds on arbitrary compact metric spaces. 
Therefore, consider a geometric separating flow $\phi$. 
By Theorem \ref{teoSepWan} we have that $\phi$ has no wandering points.
By Lemma \ref{toroNoG-exp} we know that $S$ is not the torus and
by Lemma \ref{lemaPer} we have that $\phi$ has no periodic orbits.
Now, recalling that the set of singular points is finite we apply Theorem 6.7 in \cite{Artigue} to conclude that $\phi$ is geometric expansive.
\end{proof}

\subsection{Strong kinematic expansive and strong separating flows on surfaces}
\label{skefos}
In this section we prove the second equivalence of Table \ref{HierOnSurf}.

\begin{thm}\label{skexpsurf}
Let $S$ be a compact surface and let $\phi$ 
be a continuous flow on $S$. 
The following statements are equivalent:
\begin{enumerate}
 \item $\phi$ is strong kinematic expansive,
 \item $\phi$ is strong separating,
 \item the singular points are saddles and the union of their separatrices is dense in $S$.
\end{enumerate}
\end{thm}

\begin{proof}
($1\to 2$). It holds in the general setting of compact metric spaces. 

($2\to 3$). 
By Theorem \ref{teoSepWan} we have that $\phi$ has no wandering points.
Therefore there are no parabolic sectors and singularities are of saddle type.
By Lemma \ref{lemaPer} we have that 
strong separating flows have no periodic orbits. 
So, as in Proposition 4 of \cite{Artigue} we conclude that the union of the separatrices is 
dense in $S$ given that $\phi$ is strong separating. 
This proves that (2) implies (3). 

($3\to 1$). By Theorems 6.1 and 5.3 in \cite{Artigue} we only have to consider the case 
where $S$ is a torus and the flow is minimal with a finite number of 
0-saddles. 
Let $\gamma$ be a global transversal to the flow. 
Denote by $T$ the return time function of $\gamma$ that is defined in $\gamma\setminus A$ where $A$ is a finite set.
In the points of $A$ the map $T$ diverges. 
Now take $u\neq v$ in $\gamma\setminus A$. 
Let $f\colon \gamma\to\gamma$ be the extended return map 
(notice that the points in $A$ does not return to $\gamma$ but this map can be extended by continuity to a minimal rotation $f$).
Since $A$ is finite, there is $a\in A$ such that if $f^{n_k}(u)\to a$ then $f^{n_k}(v)\to b$ with $b\notin A$. 
This implies that the flow will separate $u$ and $v$ (see the techniques of Proposition \ref{kexpsuspension} below).
\end{proof}

\begin{rmk}
% With the techniques of \cite{Artigue} it can be proved that 
The only surface admitting a strong kinematic expansive flow that is not geometric expansive is the torus. 
\end{rmk}

\begin{rmk}
Every strong kinematic expansive flow of a compact surface is topologically equivalent with a $C^\infty$ flow. 
This can be proved with Gutierrez's smoothing results as done in \cite{Artigue} 
for the geometric expansive case.
\end{rmk}

\begin{df}
If $\phi$ is a strong kinematic expansive flow we say that the expansive constant in \emph{uniform} 
if for all $\beta>0$ there is $\delta>0$ such that if $\psi$ is a time change of $\phi$ and 
$\dist(\psi_t(x),\psi_t(y))<\delta$ for all $t\in\R$ then $x,y$ are in a orbit segment of diameter smaller than $\beta$.
\end{df}

Uniformity of the expansive constant means that there is an expansive constant working for every time change.

\begin{rmk}
From the arguments above we have that on surfaces every strong kinematic expansive flow has a uniform expansive constant. 
\end{rmk}

\section{Suspension flows}
\label{secSuspension}
Let $\phi$ be a continuous flow without singularities defined on a compact metric space $X$.
A compact subset $l\subset X$ is a \emph{global section} for $\phi$
if for all $x\in l$ there is a neighborhood $U$ of $x$ such that 
$U\cap l$ is a local cross section in the sense of Whitney \cite{W} (see also \cite{BW})
and every orbit cuts $l$. 
If $\phi$ admits a global section $l\subset X$ we can consider 
the first return map
$f\colon l\to l$ satisfying $f(x)=\phi_t(x)$ if $t>0$ and $\phi_{(0,t]}(x)\cap l=\{f(x)\}$ for all $x\in l$. 
In this case we say that $\phi$ is a
\emph{suspension} of $f$.

\subsection{Kinematic expansive suspensions}

The expansiveness of a homeomorphism is known to be equivalent with the geometric expansiveness of each suspension (see \cite{BW}) and 
also to the kinematic expansiveness of a suspension of constant time (see \cite{KS79}).

Here we consider the kinematic expansiveness of a suspension with arbitrary (continuous) return time.

\begin{prop}
\label{kexpsuspension}
Suppose $\phi$ is a suspension of $f\colon l\to l$ and let $T_k\colon
l\to\R$ be such that for all $k\in\Z$ and $x\in l$,
$T_k(x)<T_{k+1}(x)$ and $\phi_{T_k(x)}(x)=f^k(x)$. Then the
following statements are equivalent:
\begin{enumerate}
\item The flow $\phi$ is kinematic expansive.
\item There is $\delta>0$ such that if
$\dist(\phi_t(x),\phi_t(y))<\delta$ for all $t\in \R$ with $x,y\in
l$ then $x=y$.
\item There is $\rho >0$ such that if
$x,y\in l$, $\dist(f^n(x),f^n(y))<\rho$ and
$|T_n(x)-T_n(y)|<\rho$ for all $n\in\Z$ then $x=y$.
\end{enumerate}
\end{prop}

\begin{proof}
(1 $\rightarrow$ 2). Let $\epsilon>0$ be such that if $x\in l$ and
$0<|s|<\epsilon$ then $\phi_s(x)\notin l$. Since $\phi$ is
kinematic expansive there is an expansive constant $\delta >0$
associated to $\epsilon$. Take $x,y\in l$ such that
$\dist(\phi_t(x),\phi_t(y))<\delta$ for all $t\in\R$. Then there
exists $s\in(-\epsilon,\epsilon)$ such that $y=\phi_s(x)$. But this
implies that $s=0$ and $x=y$.

\par (2 $\rightarrow$ 3). Let $T^*=\max\{T_1(x):x\in l\}$. The
continuity of the flow implies that there exists $\delta'>0$ such
that:
\begin{equation}
 \label{equivsuspcond}
 \hbox{if $\dist(x,y)<\delta'$ then
$\dist(\phi_t(x),\phi_t(y))<\delta$ for all $t\in [0,T^*]$. }
\end{equation}
By the
triangular inequality we have that:
\begin{equation}
 \label{equivsusp}
\dist(\phi_{T_k(x)}(x),\phi_{T_k(x)}(y))\leq \dist
(f^k(x),f^k(y))+\dist(\phi_{T_k(x)}(y),\phi_{T_k(y)}(y))
\end{equation}
for all $x,y\in l$ and $k\in\Z$. We will show that $\rho=\delta'/2$ satisfies the thesis. 
Assume that $x,y\in l$, $\dist(f^n(x),f^n(y))<\rho$ and
$|T_n(x)-T_n(y)|<\rho$ for all $n\in\Z$. 
By inequality (\ref{equivsusp}) we have that 
$\dist(\phi_{T_k(x)}(x),\phi_{T_k(x)}(y))\leq\delta'$ for all $n\in\Z$.
Now, applying condition (\ref{equivsuspcond}) we have that 
$\dist(\phi_t(x),\phi_t(y))<\delta$ for all $t\in\R$ and therefore, $x=y$ because $x,y\in l$.

(3 $\rightarrow$ 1) Given $\epsilon>0$ consider $\delta>0$ such that if $\dist(x,y)<\delta$ with $x\in l$ and $y\in X$ then 
\begin{equation}
 \label{conditionkexp}
 \hbox{there is a unique } s\in\R \hbox{ such that } |s|<\epsilon, |s|<\rho\hbox{ and } \phi_s(y)\in l\cap B_\rho(x).
\end{equation}
This value of $s$ will be denoted as $s_x(y)$ 
and we define the projection $\pi_x\colon B_\rho(x)\to l$ as $\pi_x(y)=\phi_{s_x(y)}(y)$.
We will show that $\delta$ is an expansive constant associated to $\epsilon$. 
Suppose that $\dist(\phi_t(x),\phi_t(y))<\delta$ for all $t\in\R$. 
Without loss of generality we assume that $x\in l$. 
Define the sequence $y_n=\phi_{T_n(x)}(y)$ for $n\in\Z$. 
We have that 
$f^n(y_0)=\pi_{f^n(x)}(y_n)$ and also 
$\dist(f^n(x),y_n)<\delta$ for all $n\in\Z$. 
By condition (\ref{conditionkexp}) for each $n\in\Z$ there is $s_n$ such that $|s_n|<\rho$,
$\phi_{s_n}(y_n)=f^n(y_0)$ and $\dist(f^n(x),f^n(y_0))<\rho$ for all $n\in\Z$. 
If we apply our hypothesis to the points $x,y_0\in l$, 
noting that $|s_n|=|T_n(x)-T_n(y_0)|$, we conclude that 
$x=y_0$. Therefore $x=\phi_{s_0}(y)$, and since $|s_0|<\epsilon$ by (\ref{conditionkexp}), the proof ends.
\end{proof}

As an application of this result we have that the flow on Example \ref{annulus} (periodic band) 
is kinematic expansive. Note that this is a suspension of the identity map of an arc under an 
increasing return time function. 
In the next section we will prove that the interval is the only connected space whose identity map admits 
a kinematic expansive suspension.

\subsection{Suspensions of the identity map}

In general topology it is an important task to give intrinsic topological characterizations of topological spaces. 
For example, it is known that a compact metric space $X$ is homeomorphic to the usual Cantor set if and only if 
it is totally disconnected (every component is trivial) and perfect (no isolated points). 
From a dynamical viewpoint it is also possible to characterize topological spaces. 
Let us mention, as an example, that a compact surface is a torus if and only if it admits an Anosov diffeomorphism. 
Finite sets can be characterized as those admitting a positive expansive homeomorphism. 

In this section we give a dynamical characterization of compact metric spaces that can be embedded in $\R$. 
In order to obtain this kind of result we recall a topological characterization of such spaces.

\begin{thm}
\label{teoRudin}
 A compact metric space $l$ is homeomorphic to a subset of $\R$
 if and only if the following statements hold:
\begin{enumerate}  
 \item the components of $l$ are points or compact arcs,
 \item no interior point of an arc-component $a$ is a limit point of $l\setminus a$ and 
 \item each point of $l$ has arbitrarily small neighborhoods whose boundaries are finite sets.
 \end{enumerate}
\end{thm}

See \cite{Rudin} for a proof.

\begin{thm}
\label{carCompR}
If $l$ is a compact metric space then the following statements are equivalent:
\begin{enumerate}
 \item the identity map of $l$ admits a kinematic expansive suspension,
 \item there is a continuous and locally injective map $T\colon l\to \R$, i.e., there is $\delta>0$ such that if $0<\dist(x,y)<\delta$ then $T(x)\neq T(y)$ and
 \item $l$ is homeomorphic to a subset of $\R$.
\end{enumerate}

\end{thm}

\begin{proof}
($1\to 2$) The return time map $T$ making the suspension of the identity map of $l$ kinematic expansive, has to be locally injective by 
Proposition \ref{kexpsuspension} (item 3). 

($2\to 3$) Since $l$ is compact we have that $T$ is a local homeomorphism. 
Therefore $l$ satisfies item (3) of Theorem \ref{teoRudin}. 
To prove the first item, consider a non trivial component $a$ of $l$. 
As we mentioned, $T$ is a local homeomorphism, therefore $a$ is a compact connected one-dimensional manifold. 
If $a$ is not a compact arc, then it must be a circle, but this easily gives us that $T$ cannot be locally injective. 
Therefore item (1) holds. 
The second item of Theorem \ref{teoRudin} follows again because $T$ is a local homeomorphism. 

($3\to 1$) Let $T\colon l\to \R^+$ be an embedding of $l$. 
Applying Proposition \ref{kexpsuspension} we have that the suspension of the identity of $l$ under $T$ is kinematic expansive.
\end{proof}

Let $\Gamma$ be the set of periodic orbits of $\phi$ endowed with the relative topology 
induced by the Hausdorff distance between compact subsets of $X$. 
Recall that 
\[
 \dist_H(A,B)=\inf \{\epsilon>0:B\subset B_\epsilon(A), A\subset B_\epsilon(B)\}
\]
is the Hausdorff distance between the compact sets $A,B\subset X$.
Let $T\colon\Gamma\to\R^+$ be the period function defined such that $T(\gamma)$ is the period of the periodic orbit $\gamma$.
The following proposition gives another characterization of the suspensions of the previous theorem.

\begin{lem}
\label{lemContPer}
 If $\phi$ is a kinematic expansive on a compact metric space 
 then the period function $T$ is continuous.
\end{lem}

\begin{proof}
Let $\gamma_n$ be a sequence of periodic orbits converging in the Hausdorff distance to a periodic orbit $\gamma$. 
Let $l$ be a local cross section through a point $p\in\gamma$. 
If $T(\gamma_n)$ do not converge to $T(\gamma)$ then $\gamma_n$, for large $n$, must meet at least twice to $l$, say in $x_n$ and $y_n$. 
Therefore, $x_n$ and $y_n$ contradict the kinematic expansiveness of $\phi$.
\end{proof}

\begin{prop}
\label{K-expTodoPer}
 Suppose that $\phi$ is a kinematic expansive flow 
 without singularities on a compact metric space such that every orbit is compact. 
 Then it is a suspension of the identity map of a compact subset of $\R$.
\end{prop}

\begin{proof}
 By Lemma \ref{lemContPer} and Theorem \ref{carCompR} we have that every point has a local cross section homeomorphic to a compact subset of $\R$. 
 Then every point admits a compact local cross section $l$ such that $\phi_\R(l)$ is an open subset of $X$. 
 With the techniques of \cite{BW} it is easy to prove that $\phi$ is a suspension. 
 Therefore we conclude by Theorem \ref{carCompR}.
\end{proof}

\subsection{Arc homeomorphisms}
\label{subsecArc}
In this section we study when a homeomorphism of a compact arc $I$ admits a kinematic expansive suspension. 
We consider homeomorphisms of class $C^0$ and $C^1$.
We say that $\phi$ is a \emph{semi-flow} if it is a continuous partial action of $\R^+$.
Recall that the $\omega$-\emph{limit set} of $x$ is 
$$\omega(x)=\{y:\exists t_k\to+\infty \hbox{ such that } \phi_{t_k}(x)\to y\hbox{ as }k\to+\infty\}.$$
\begin{lem}
\label{lemaDisco-Anillo}
 Let $\phi$ be a continuous semi-flow on an annulus $A$ such that one component of the 
 boundary is a periodic orbit $\gamma$, the other component is transversal to the flow and 
 the $\omega$-limit set of every point in $A$ is $\gamma$. 
 Then $\phi$ admits a kinematic expansive time change.
\end{lem}

\begin{proof}
Consider a global cross section $l$, as in Figure \ref{figAnilloSemiKexp}, and identify $l$ with the interval $[0,1]$. 
The return map to $l$ is conjugated with
$f\colon [0,1]\to[0,1]$ defined by 
$f(x)=x/2$. 
Then
$f^n(x)=x/2^n$ for all $n\geq0$ and $x\in [0,1]$. 
Define $a_n=f^n(1)$ and $b_n= f^n(1/2+1/2^{n+2})$.
In this way we have that $b_n\in(a_n,a_{n+1})$ for all $n\geq 0$.
Define $T\colon [0,1]\to\R$ by $T(a_n)=T(1)=T(0)=1$ and $T(b_n)=1+1/(n+1)$ for all $n\geq 0$ and
extended by linearity in $(a_n,b_n)$ and $(b_n,a_{n+1})$ for all $n\geq 0$. See Figure \ref{figMapaT}.

\begin{figure}[htbp]
\begin{center}
\includegraphics{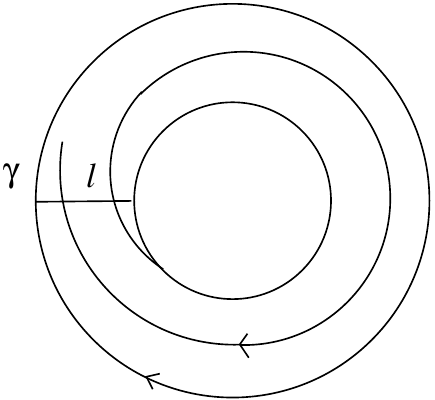}
    \caption{.}
   \label{figAnilloSemiKexp}
\end{center}
\end{figure}

Consider a semi-flow $\psi$, a time change of $\phi$ with returning time $T$ to the section $l$.
We will show that $\psi$ is kinematic expansive.
\begin{figure}[htbp]
\begin{center}
\includegraphics{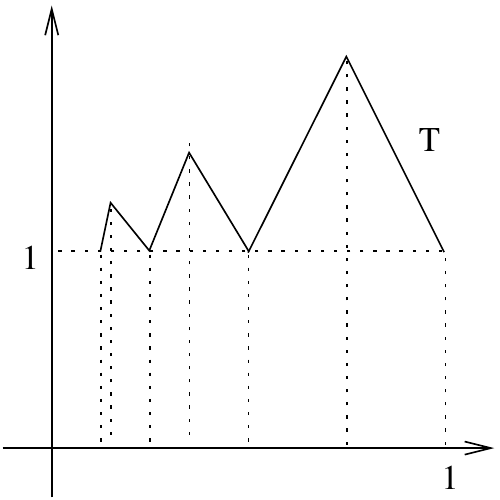}
    \caption{The return time function $T$.}
   \label{figMapaT}
\end{center}
\end{figure}
Every point in $\gamma$ is separated from any other outside $\gamma$, as can be easily seen. 
We now study two cases, taking $x,y\in l=[0,1]$, $x\neq y$.

Case 1: $a_1<x<y\leq a_0$. Notice that there exists $n_0$ such that if $n\geq n_0$ then 
$a_{n+1}<b_n<x_n<y_n<a_n$, being $x_n=f^n(x)$ and $y_n=f^n(y)$.
Then for all $n\geq n_0$ we have that: 
$$T(x_n)-T(y_n)\geq \frac{(y_n-x_n)n^{-1}}{a_{n+1}-a_n}=
\frac{\left(\frac{y-x}{2^n}\right)n^{-1}}{1/2^{n+1}}=2(y-x)/n.$$ 
And then $\sum_{i=0}^\infty T(x_i)-T(y_i)=+\infty$ and therefore the points $x,y$ are separated by the flow $\psi$.  

Case 2: $a_0< x\leq a_1<y<a_2$. 
From the definition of $T$ it is easy to see that 
$$T(x_n)=\frac{2(a_1-x)}{n(1-1/2^{n+1})}$$ 
and $$T(y_n)=\frac{2(a_2-y)}{(n+1)(1-1/2^{n+2})}.$$
Let $\alpha_n=\frac{2(a_1-x)}{1-1/2^{n+1}}$ and $\beta_n=\frac{2(a_2-y)}{1-1/2^{n+2}}$. Then 
$$T(x_n)-T(y_n)=\frac{\alpha_n}{n} - \frac{\beta_n}{n+1}=\frac{\alpha_n-\beta_n}{n+1} -\frac{\alpha_n}{n(n+1)}$$
Again we have that
$\sum_{n=0}^\infty T(x_n)-T(y_n)=+\infty$ since $\alpha_n-\beta_n\to 1/2 + 2(y-x) \neq 0$.
\end{proof}

Recall that for an arc homeomorphism preserving orientation, the periodic points are in fact fixed points, and given any closed set $F$ 
of the arc 
there is a preserving orientation homeomorphism whose set of fixed points is $F$.
If $f$ reverses orientation we have that there is a unique fixed point and other periodic points have period 2. 

\begin{prop} 
\label{arcSusp}
A homeomorphism $f\colon I\to I$ admits a kinematic expansive suspension if and only if 
the set of periodic points has finitely many components and the period function is continuous (i.e. in the reversing orientation case, 
the fixed point is not accumulated by points of period 2).
\end{prop}

\begin{proof}
 ($\Rightarrow$)
 Let us start assuming that $f$ admits a kinematic expansive suspension.
 Suppose first that $f$ reverses orientation. 
 As we said $f$ has a unique fixed point $p$. 
 Now, it is easy to see that $x$ and $f(x)$ contradicts expansiveness if $x$ is a periodic point (of period 2) arbitrarily close to $p$.
Assume now that there are infinitely many wandering components. 
We have that for all $\epsilon>0$ there is a wandering point $x$ such that $\dist(x,f^n(x))<\epsilon$ for all 
$n\in\Z$. 
Consider a time map $T\colon I\to \R^+$. 
Since it is uniformly continuous we have that for all $\delta>0$, the value of $\epsilon$ can be chosen in such a way that 
if $\dist(x,y)<\epsilon$ then $|T(x)-T(y)|<\delta$. Therefore, the points $x$ and $f(x)$ contradicts the expansiveness.

 ($\Leftarrow$) On each component of fixed points consider an increasing time map. 
 On wandering points use Lemma \ref{lemaDisco-Anillo}.
\end{proof}

The smooth case is very restrictive as the following result shows.

\begin{prop}
\label{suspKexpIC1}
 Assume that $f\colon I\to I$ is a homeomorphism and $T\colon I\to\R^+$ is $C^1$. 
 If the suspension $(f,T)$ is kinematic expansive then $f$ is the identity and $T$ is strictly increasing or decreasing.
\end{prop}

\begin{proof}
 Let us assume first that $f$ is increasing.
 By contradiction assume that it is not the identity, 
 therefore there are two fixed points $p,q\in I$ 
 such that for all $x\in(p,q)$ we have that $f^n(x)\to q$ and $f^{-n}(x)\to p$ as $n\to\infty$.
Since $T$ is smooth we have that 
$T(y)-T(x)=\int_x^y T'(u)\,du$. 
Therefore, taking $x,y\in(p,q)$ arbitrarily close we can easily contradict Proposition \ref{kexpsuspension}. 

Assume now that $f$ is
decreasing and take the fixed point $p$ of $f$. 
If close to $p$ there are wandering points then we can arrive to a contradiction as in the previous case. 
The other possible case is that every point close to $p$ is periodic with period 2. 
If $x$ is close to $p$ and $y=f(x)$ it is easy to see that $x,y$ contradicts the expansiveness of the suspension flow. 
This contradiction proves that $f(x)=x$ for all $x\in I$. 

Now applying Proposition \ref{kexpsuspension} we see that $T$ must be increasing or decreasing.
\end{proof}

\subsection{Circle homeomorphisms}

Let $f\colon S^1\to S^1$ be homeomorphism of the circle. 
Recall that if the are no wandering points then it is conjugated to a rotation. 
In other case we say that the wandering set of $f$ is \emph{finitely generated} if 
there is a finite number of disjoint open arcs $a_1,\dots,a_n$ such that the wandering set is the union 
$$
\bigcup_{j\in\Z,i=1,\dots,n} f^j(a_i).
$$

In the following Theorem we exclude the case where $f$ is minimal because we have no $C^0$ general answer.

\begin{thm}
\label{teoSusCir}
 A non-minimal circle homeomorphism $f\colon S^1\to S^1$ preserving orientation 
 admits a kinematic expansive suspension if and only if 
 its wandering set is non-empty and finitely generated.
\end{thm}

\begin{proof}
($\Rightarrow$) Assume that $f$ admits a kinematic expansive suspension.
If $f$ has no wandering points then it is a rotation, and since it is not minimal, it is a periodic (rational) rotation. 
Now it is easy to see that there are arbitrarily close points with the same period (for the flow) contradicting expansiveness. 
Therefore the wandering set is not empty.
The wandering set is finitely generated by the arguments in the proof of Proposition \ref{arcSusp}.

% \emph{Converse}.
($\Leftarrow$)
Now assume that the wandering set is generated by one interval (it is easy to extend the proof to the general case).
It is known that
$f\colon \Omega\to\Omega$ is an expansive homeomorphism, where $\Omega$ denotes the non-wandering set of $f$.
Assume that the wandering set is the disjoint union $\cup_{n\in\Z} f^n(I)$ where $I=(a,b)$ is an open arc. 
Without loss of generality we will assume that 
\begin{equation}
 \label{denjoydist}
 \dist(f^n(x),f^n(y))=\frac{\dist(x,y)}{2^n}
\end{equation}
for all $x,y\in I$ and $n\geq 0$. 
For each $n\geq 0$ take a point $z_n\in f^n(I)$ such that 
$f^{-n}(z_n)\to a$. 

Define a continuous map $T\colon S^1\to \R^+$, the return time function, as 
$T(x)=1$ if $x\in\Omega$ or $x\in \cup_{n\geq 0} f^{-n}(I)$, 
$T(z_n)=1+1/n$ for all $n> 0$ and extend $T$ linearly on each $f^n(I)$ with $n> 0$.

We claim that the flow on the torus with return map $f$ and return time $T$, defined above, is kinematic expansive.
 To prove kinematic expansiveness we will use item (3) of Proposition \ref{kexpsuspension}. 
 We know that $f\colon\Omega\to\Omega$ is an expansive homeomorphism. 
 It is easy to see that if $x\in I$ and $y\notin \clos(I)$ then $x,y$ are separated by $f$. 
 It only rests to consider $x,y\in\clos(I)$. 
 We divide the proof in two cases. 
 
 First suppose that $x,y\in I$. In the arc $I$ we consider an order such that $a<b$ and using the homeomorphism $f$ we induce an 
 order on each $f^n(I)$ with $n\in\Z$. Assume that $x<y$.
 Recall that the sequence $z_n\in f^n(I)$ 
 used to define the return time $T$ has the property $f^{-n}(z_n)\to a$. Therefore there is $n_0\geq 0$ such that 
 $z_n<f^n(x)<f^n(y)$ for all $n\geq n_0$. 
 Let us introduce the notation $x_n=f^{n+n_0}(x)$ and $y_n=f^{n+n_0}(y)$. 
 By the definition of $T$ (recall that it was extended linearly) and equation (\ref{denjoydist}) we have that 
 \[
 \begin{array}{ll}
  T(x_n)-T(y_n) & \displaystyle\geq \frac{\dist(x_n,y_n)}{(n_0+n)\dist(f^{n+n_0}(a),f^{n+n_0}(b))}\\
   & \displaystyle= \frac{\dist(x_{n_0},y_{n_0})}{(n_0+n)\dist(f^{n_0}(a),f^{n_0}(b))}
 \end{array}
 \]
for all $n\geq 0$. Then 
\[
 \sum_{n\geq 0} T(f^n(x))-T(f^n(y))=\infty.
\]
 
 Now assume that $x=a$ (a extreme point of $I$) and $y\in I$. 
 Assume that $z_n<f^n(y)$ for all $n\geq n_0$.
 As before it can be proved that 
 \[
  T(y_n)-T(x_n)\geq \frac{1}{n}\frac{\dist(f^{n_0}(b),f^{n_0}(y))}{\dist(a,b)}.
 \]
 And we arrive again to a divergent series. 
 The case $y=b$ is similar to this case. 
 This proves that the flow is kinematic expansive.
\end{proof}

\begin{thm}
 A reversing orientation homeomorphism $f\colon S^1\to S^1$ admits a kinematic expansive suspension if 
 and only if it has wandering points, fixed points are not accumulated by periodic points and the wandering set has 
 a finite number of components.
\end{thm}

\begin{proof}
 Since $f$ reverses orientation it has two fixed points. 
 The dynamics is then reduced to an interval homeomorphism and we can apply Proposition \ref{arcSusp} to conclude the proof.
\end{proof}

\subsection{Smooth suspensions of circle diffeomorphisms}
In this section we apply the results of \cite{AvilaKocsard} to 
study smooth kinematic expansive suspensions of irrational rotations.

\begin{thm}\label{NoIrrK-exp}
No suspension of an irrational rotation $f\colon S^1\to S^1$ 
with $C^1$ return time function $T\colon S^1\to\R^+$ is kinematic expansive.
\end{thm}

\begin{proof}
Let $\mu$ denote the $f$-invariant Lebesgue probability measure on the circle. 
Define
\[
 \tau=\int_{S^1} T\,d\mu.
\]
Denote by $\alpha\in\R\setminus\Q$ the angle of the rotation $f$. 
Let $q_n\in\N$ be the denominator of a rational approximation of $\alpha$. 
It holds that $f^{q_n}(x)\to x$ as $n\to\infty$ for all $x\in S^1$.
See, for example, Section 2.3.2 of \cite{AvilaKocsard} for more details. 
Consider the Birkhoff sum
\[
 T_m(x)=\sum_{i=0}^{m-1} T(f^i(x)).
\]
The improved Denjoy-Koksma Theorem proved in \cite{AvilaKocsard} states that 
\begin{equation}
 \label{desAK}
 \sup_{x\in S^1} \left|\tau q_n-T_{q_n}(x)\right|\to 0
\end{equation}
as $n\to\infty$. 
Fix $x_0\in S^1$ and define $x_n=f^n(x_0)$ for all $n\geq 0$.
It is easy to see that $T_{m+n}(x)=T_m(x_n)+T_n(x_0)$. 
Then 
\[
 T_n(x_m)-T_n(x_0)=T_m(x_n)-T_m(x_0)
\]
and in particular
\[
 T_{q_n}(x_m)-T_{q_n}(x_0)=T_m(x_{q_n})-T_m(x_0)
\]
for all $m,n\geq 0$. Applying equation (\ref{desAK}) we have that for all $\epsilon>0$ there is $N$ 
such that 
\[
 |T_{q_n}(x_m)-T_{q_n}(x_0)|<\epsilon
\]
for all $n\geq N$ and $m\geq 0$.
Therefore 
\[| T_m(x_{q_n})-T_m(x_0)|<\epsilon\]
for all $n\geq N$ and $m\geq 0$. 
Now we can take $n\geq N$ such that the distance between $x_{q_n}$ and $x_0$ is smaller than $\epsilon$. 
Therefore these two points are not separated by the suspension flow in positive time. 
Arguing in the same way for negative time we conclude that this flow is not kinematic expansive. 
\end{proof}

\begin{question}
Are there $C^0$ minimal kinematic expansive flows on the torus? 
\end{question}

\section{Kinematic expansive flows on surfaces}
\label{secKinSurf}

% \aca Se me ocurre un nuevo objetivo para esta seccion. Poner como hipotesis 
% que la superficie es una esfera menos una cantidad finita de discos y caracterizar 
% cuando un flujo admite una reparametrizacion cinematicamente expansiva. 
% En estas superficies no hay recurrencia no trivial y la cosa es mas sencilla.

In Section \ref{HieOnSurf} we studied flows with the property of having every time change being 
kinematic expansive (strong kinematic expansiveness). 
In this section we consider what could be called \emph{conditional expansiveness}: 
the kinematic expansiveness of the flow depends on the time change. 
We consider flows on the disc and the annulus. 
In the final subsection we prove that every compact surface admits a kinematic expansive flow.

\subsection{The disc}
Let $D$ be a two-dimensional compact disc and consider $\phi\colon \R\times D\to D$ a continuous flow.
It is well known that under these conditions, $\phi$ has a singular point. 
For a kinematic expansive flow we show that at least one singularity must be in the interior of the disc. 
Next we study the relationship between the number of singularities and the differentiability of the flow. 

\begin{prop}
\label{lemaDiscoSing}
 If $\phi$ is a kinematic expansive flow on a disc $D$ then $\phi$ has a singularity in the interior of $D$.
\end{prop}

\begin{proof}
Assume by contradiction that the singularities are in the boundary $\partial D$. 
The $\alpha$ and $\omega$-limit set of every point of $D$ must be a singular point, it follows by Poincar\'e-Bendixon Theorem. 
But this implies that there are arbitrarily small 
loops associated to a singular point (elliptic sectors in the terminology of \cite{Hartman}). 
This contradicts kinematic expansiveness because such singularities in the boundary cannot be $\phi$-isolated.
\end{proof}

The following result proves that the disc admits kinematic expansive flows. 
In particular this flow may have just one singular point.

\begin{prop}
\label{propDisco} 
Suppose that $\phi$ is a continuous flow in $D$ with a finite number of singularities and $p\in D$ is an interior point.
Assume that $p$ is a repeller fixed point and for all $x\neq p$ interior to $D$ the $\omega$-limit set of $x$ is 
$\omega(x)=\partial D$. Then $\phi$ admits a kinematic expansive time change.
\end{prop}

\begin{proof}
Let $l$ be a local cross section of the flow meeting the boundary of $D$.
Suppose that the return map on $l$ is the continuous map $f\colon l\to l$ 
and the return time is $T\colon l\to\R^+$. 
If there are no singular point in the boundary then we can 
apply Lemma \ref{lemaDisco-Anillo} to conclude.
Therefore we will assume that there are singularities in the boundary and $l$ is as in Figure \ref{disc}, where 
$q$ is another singular point.
\begin{figure}[htbp]
\begin{center}
\includegraphics{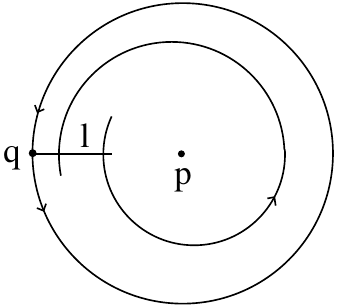}
    \caption{Kinematic expansive flow on the disc.}
   \label{disc}
\end{center}
\end{figure}

Without loss of generality we assume that the return map $f$ is $f(x,0)=(x/2,0)$. 
Consider a time change such that the return map of $l$ is $T(x,0)=1/x$.
Given $1/2\leq x<y<1$ we define $x_n=f^n(x)=x/2^n$ and
$y_n=f^n(y)=y/2^n$. 
Then $$T(y_n)-T(x_n)=2^n\left(\frac 1{x} - \frac 1{y}\right)$$
and therefore $\sum_{i=0}^\infty T(y_i)-T(x_i)=+\infty$. It implies that $\phi$ is kinematic expansive.
\end{proof}

The previous result does not hold if we add a hypothesis of differentiability.

\begin{thm}
If $\phi$ is a smooth kinematic expansive flows in the disc the $\phi$ has at least two singular points.
% with just one singular point.
\end{thm}

\begin{proof} 
By contradiction assume that $\phi$ has only one singularity $p\in D$.
By Proposition \ref{lemaDiscoSing} we know that the singular point is in the interior of $D$ 
and therefore $\partial D$ is a periodic orbit.
Since there is just one singular point, the periodic orbits in $D$ can be totally ordered with respect to 
the interior singular point (i.e., if $\gamma_1,\gamma_2$ are periodic orbits then $\gamma_1<\gamma_2$ 
if $\gamma_1$ separates $p$ from $\gamma_2$). 
Considering a minimal periodic orbit, we obtain a sub-disc $D'\subset D$, bounded by such minimal periodic orbit, 
such that in the interior of $D'$ there is no periodic orbit. 
Now, applying the techniques of Proposition \ref{suspKexpIC1}, near $\partial D'$, 
we arrive to a contradiction.
% conclude the proof.
\end{proof}

\subsection{Periodic bands} 

Denote by $A\subset \R^2$ a compact annulus bounded by two circles centered at the origin. 

\begin{prop}
\label{nonExp1}
 Suppose that $\phi$ is a kinematic expansive flow on $A$ 
 such that every orbit is contained in a circle centered at the origin. 
 If $\phi$ has singular points then they all are in one of the components of the boundary.
 In particular there are no interior singular points.
 \end{prop}

\begin{proof}
 We know that the set of singular points is finite. 
 Let us first show that there is no singularity in the interior. 
 We argue by contradiction.
 Take a segment $l$ transversal to the flow meeting at $p$ the circle of an interior singularity. 
 We can assume that there are no singular points in the circles of any $q\in l$ if $q\neq p$. 
 Since the circle of $p$ has at least one singularity we have that the return time map of $l\setminus\{p\}$ 
 diverges to $+\infty$ at $p$. 
 Therefore we can find two points, as close to $p$ as we wish, in different components of $l\setminus\{p\}$ with the same period. 
 These points contradict kinematic expansiveness. 
 
 Now assume that there are singular points in both components of $\partial A$. 
 Let $s$ be a global cross section of the flow meeting once each interior orbit. 
 As before, the return time map $T$ diverges in the boundaries of $s$. 
 Since $T$ is continuous, it has a minimum at some interior point $x\in s$. 
 Now we can find two points in different components of $l\setminus \{x\}$ with the same period. 
 If these points are sufficiently close to $x$, then kinematic expansiveness can be contradicted for arbitrary small 
 expansive constants.
\end{proof}

\begin{rmk}
 Notice that we have considered kinematic expansive flows on the annulus in Section \ref{subsecArc} (i.e., 
 suspensions of increasing arc homeomorphisms). 
\end{rmk}

% \begin{prop}
% \label{nonExp2}
%  Suppose that $\phi$ is a kinematic expansive flow on a surface $S$ such that every orbit is compact. 
%  Then $S$ is homeomorphic to an annulus. 
% \end{prop}
% 
% \begin{proof}
% It is an easy consequence of Proposition \ref{K-expTodoPer}.
% \end{proof}

\subsection{Every compact surface admits a kinematic expansive flow}
By our previous results we have that the sphere (and surfaces do not admitting non-trivial recurrence) 
does not admit strong kinematic expansive flows. 
For kinematic expansiveness there is no such restriction.

\begin{thm}
Every compact surface admits a kinematic expansive flow.
\end{thm}

\begin{proof}
Given a compact surface $S$ consider a triangulation $T_1,\dots,T_n$. 
Fix an orientation on each edge. In Figure \ref{triangulos} we see that each triangle $T_i$ admits a 
kinematic expansive flow (recall Proposition \ref{propDisco}) 
with any prescribed 
orientation in the edges and singular points in the corners. 
Now it is easy to see that the global flow is kinematic expansive.
\begin{figure}[htbp]
\begin{center}
\includegraphics{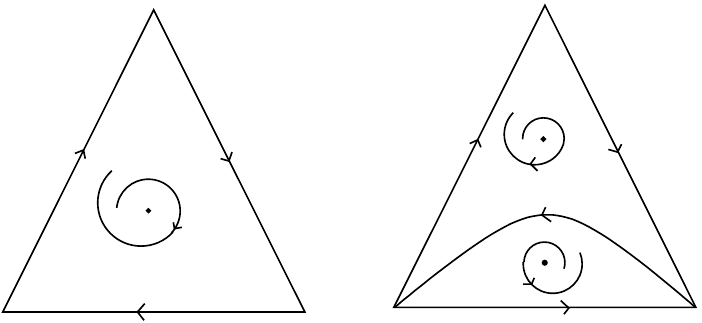}
    \caption{A kinematic expansive flow on each triangle.}
   \label{triangulos}
\end{center}
\end{figure}
\end{proof}

\section{Positive expansive flows} 
\label{secPositive}
In this section we consider positive expansive flows. 

\begin{df} 
A flow $\phi$ is \emph{positive kinematic expansive} if for all $\epsilon>0$ there exists $\delta>0$ such that if 
$\dist(\phi_t(x),\phi_t(y))<\delta$ for all $t\geq0$ then there exists $s\in\R$ such that $y=\phi_s(x)$ and $|s|<\epsilon$.
\end{df}

In \cite{Artigue2} it is proved that positive geometric expansive flows 
are trivial, they consist on a finite number of compact orbits (singular or periodic).
The case of positive kinematic expansiveness is different, as the following remark shows.

\begin{rmk}
If $X$ is a compact subset of $\R$ then for every injective and continuous map $T\colon
X\to\R^+$ the suspension flow of the identity map
$f\colon X\to X$ by $T$ is positive kinematic expansive. 
Notice also that it is negative expansive, i.e., its inverse flow is positive expansive.
\end{rmk}

In this section we first study the behavior of a positive kinematic expansive (and also separating) 
flow near a compact orbit.  
On surfaces we give a characterization of such flows. 
We also show that on a compact metric space a positive kinematic expansive flow may not be negative 
kinematic expansive.

\subsection{Periodic orbits}

In this section we consider compact orbits of positive kinematic expansive and separating flows.

% 
% Let us now show how is the local behavior near a periodic orbit of a positive kinematic expansive flow.

\begin{prop}
\label{posexpper}
Let $\phi$ be a positive kinematic expansive flow on a compact metric space.
If $\gamma$
is a periodic orbit then there exists $\delta>0$ such that if
$\phi_{\R^+}(x)\subset B_\delta(\gamma)$ and $x\notin\gamma$ then $x$ is a
periodic point. 
\end{prop}

\begin{proof}
 Let $l\subset X$ be a small local cross section of time $\tau>0$ meeting 
 the periodic orbit $\gamma$ only at some point $p\in\gamma$. 
 Assume that there is $x\in l$ such that $\phi_t(x)$ is close to $\gamma$ for all 
 $t\geq 0$.
 Let $x_n$ be the sequence of returns of $x$ to $l$ and consider the increasing sequence of return times $t_n$ such that 
 $\phi_{t_n}(x)=x_n$ with $x_0=x$. 
 Take $y=x_1$. 
 Denote by $T$ the return time map of $l$. 
 Consider $k=\sup_{a,b\in l}|T(a)-T(b)|$. 
 Notice that $k\to 0$  if $\diam(l)\to 0$. 
 Denote by $f$ the first return map of $l$, $f(a)=\phi_{T(a)}(a)$ for all $a\in l$ where $f$ is defined. 
 Since $y=f(x)$ we have that
 \[
  \left|\sum_{i=0}^n T(f^i(y))-\sum_{i=0}^n T(f^i(x))\right| = |T(f^{n+1}(x))-T(x)|\leq k
 \]
 for all $n\geq 0$.

 Therefore, expansiveness implies that $x$ and $y$ are in the same local orbit 
 and since $x,y$ are in the local cross section $l$ we have that $y=x$. 
 Then $x$ is a fixed point of $f$ and a periodic point of $\phi$.
\end{proof}

\begin{df} 
A flow $\phi$ is \emph{positive separating} if there exists $\delta>0$ such that if 
$\dist(\phi_t(x),\phi_t(y))<\delta$ for all $t\geq0$ then $y=\phi_\R(x)$.
\end{df}

The following example is a positive separating flow that is not positive kinematic expansive 
and it shows that Proposition \ref{posexpper} does not hold for positive separating flows. 

\begin{ejp}
 Let $X=\{0,1\}\cup\{x_n:n\in\Z\}\subset\R$ such that $x_n$ is an increasing sequence, $\lim_{n\to-\infty} x_n=0$ and 
 $\lim_{n\to\infty} x_n=1$. Define the homeomorphism $f\colon X\to X$ by $f(0)=0$, $f(1)=1$ and $f(x_n)=x_{n+1}$. 
 Consider $T\colon X\to \R^+$ given by $T(0)=T(1)=1$ and $T(x_n)=\frac1{|n|+1}$ for all $n\in\Z$. 
 Let $\phi$ be the suspension flow of $f$ by $T$.
 By the previous proposition it is easy to see that it is not positive kinematic expansive. 
 It also holds that $\phi$ is positive separating (the proof is trivial because there are only three orbits for the flow) 
 and 
 it shows that Proposition \ref{posexpper} does not hold for positive separating flows.
\end{ejp}

\begin{prop}
\label{posepsing}
Suppose that $\phi$ is a positive separating flow with a singular point $p\in X$. 
If for some $x\in X$ it holds that $p\in \omega(x)$ then $x=p$. 
Consequently, there are no singularities in the $\omega$-limit set of a regular point.
\end{prop}

\begin{proof}
 Arguing by contradiction it is easy to see that there is $y\neq p$ such that $\omega(y)=\{p\}$. 
 But this contradicts that $\phi$ is positive separating.
\end{proof}

\subsection{Positive kinematic expansive flows on surfaces}

In this section we classify positive kinematic expansive flows of compact surfaces. 
We consider the $C^0$ and $C^2$ case.

\begin{lem}\label{peixoto2}
Let $l=[a,b]$ and $l'$ be two compact local cross sections and suppose there exists a
continuous non-bounded function $\tau\colon [a,b)\to\R$ such that
$\phi_{\tau(x)}x\in l'$ for all $x$ in $[a,b)$.
% $\lim_{x\to b}\tau(x)=+\infty$.
Then $\omega(b)\subset\sing$.
\end{lem}

\begin{proof}
See Lemma 3 in \cite{Pe} or Lemma 2.2 in \cite{Artigue}.
\end{proof}

\begin{prop}
\label{pkexpsurfsing}
If $\phi$ is positive kinematic expansive on a compact surface then $\sing(\phi)=\emptyset$.
\end{prop}

\begin{proof}
 By contradiction assume that $p\in S$ is a singular point. 
 By Proposition \ref{posepsing} we have that $p$ must be a repeller. 
 Consider the open set 
 $$U=\left\{x\in S:\lim_{t\to-\infty}\phi_t(x)=p\right\}.$$
%  $U$ of points with $\alpha$-limit $\{p\}$. 
 We will show that $\partial U$ is a periodic orbit. 
 Take $x\in U$, $x\neq p$. Consider $y\in\omega(x)$. 
 By Proposition \ref{posepsing} we have that $y$ is a regular point. 
 Let $l$ be a compact local cross section with $y$ as an extreme point. 
 Assume that $\phi_{\R^+}(x)$ cuts $l$ infinitely many times and
 denote by $x_1,x_2,\dots$ the cuts of the positive trajectory of $x$ with $l$. 
 Define 
 $$V=\{z\in l:\phi_{\R^+}(z)\cap l\neq\emptyset\}.$$
%  $V$ be the set of points $z\in l$ whose positive trajectory meets $l$. 
 We will show that the arc $[x_1,x_2]\subset l$ is contained in $V$. 
 Denote by $V_1$ the connected component of $V\cap [x_1,x_2]$ containing $x_1$. 
 We have that $V_1$ is open in $[x_1,x_2]$.
 By Lemma \ref{peixoto2} and Proposition \ref{posepsing} we have that the return time of the points in $V_1$ to $l$ is bounded. 
 Therefore, by the continuity of the flow and the compactness of $l$, the extreme points of $V_1$ are in $V_1$ and it is closed. 
 This proves that $[x_1,x_2]\subset V_1$. Analogously it can be proved that $[x_n,x_{n+1}]\subset V$. 
 Therefore every point in $l\setminus \{y\}$ returns to $l$. 
 Again, if the return time were not bounded we contradict Proposition \ref{posepsing}. 
 Therefore $y$ is a periodic point. 
 But this is a contradiction with Proposition \ref{posexpper}. 
 Then, there cannot be singular points.
\end{proof}

\begin{thm}
\label{teoPosExpSup}
Let $\phi$ be a continuous flow on a compact surface. If $\phi$ is positive kinematic expansive then it is 
topologically equivalent with one of the following models:
\begin{enumerate}
 \item A suspension of the identity of $[0,1]$. 
 \item A suspension of an orientation preserving circle homeomorphism with irrational 
 rotation number and finitely generated wandering set.
\end{enumerate}
\end{thm}

\begin{proof}

By Proposition \ref{pkexpsurfsing} we only have to consider flows without singularities.
It is known that the only surfaces admitting such flows are: the torus, the annulus, the Klein's bottle and the Moebius band. 
Suppose first that $\phi$ has a periodic orbit $\gamma$. 
If there is $x\notin\gamma$ such that $\phi_{-t}(x)\to\gamma$ as $t\to\infty$ then, arguing as in the proof of 
Proposition \ref{pkexpsurfsing}, 
% \marginpar{armar un lemma?}
we can prove that $\omega(x)$ is a periodic orbit. 
But this contradicts Proposition \ref{posexpper}. 
Therefore, every orbit close to $\gamma$ must be periodic. 
Now, applying Lemma \ref{peixoto2} we have that every orbit is periodic because there are no singular points. 
Notice that $\gamma$ must be two-sided, i.e., if $U$ is a tubular neighborhood of $\gamma$ 
then $U\setminus\gamma$ has two components. This is because, if this were not the case, then 
if $T$ is the period of $\gamma$ and $x$ is close to $\gamma$ then $x$ and $y=\phi_T(x)\neq x$ would contradict 
kinematic expansiveness. Now recall that the Moebius band and the Klein bottle always have periodic orbits. 
Therefore $S$ must be orientable. 
Also, the torus does not admit a kinematic expansive flow whit every orbit being periodic. 
Therefore $S$ must be be an annulus. 

Now suppose that $S$ is the torus and $\phi$ has no periodic orbits. 
In this case it is known that $\phi$ is a suspension. 
Thus, we conclude by Theorem \ref{teoSusCir}.
\end{proof}

\begin{thm}
 If $\phi$ is a $C^2$ positive kinematic expansive flow on a compact surface then $\phi$ is a suspension of the identity of $[0,1]$ and 
 $S$ is an annulus.
\end{thm}

\begin{proof}
By Theorem \ref{teoPosExpSup} we have to show that $\phi$ cannot satisfy item (2) in this Theorem. 
By \cite{Gutierrez} and assuming that $\phi$ satisfies item (2) we have that 
$\phi$ is a minimal flow on the torus. 
By Theorem \ref{NoIrrK-exp} we know that $\phi$ cannot be expansive. 
This contradiction proves the theorem.
\end{proof}

Let us introduce a natural definition.

\begin{df}
 A flow is \emph{positive strong kinematic expansive} if every time change is positive kinematic expansive.
\end{df}

An example of such flow is the horocycle flow of a surface of negative curvature, this is proved in \cite{Gura}. 
The horocycle flow is defined on three-dimensional manifold. We now apply our results to conclude that such flows do not exist 
on surfaces.

\begin{cor}
 There are no positive strong kinematic expansive flows of surfaces.
\end{cor}

\begin{proof}
 It follows by Proposition \ref{pkexpsurfsing} and Theorem \ref{skexpsurf}.
\end{proof}

\subsection{Minimal positive expansiveness}
In this section we consider an adaptation of an example in \cite{KS} to show 
that minimal positive kinematic expansive flows may not be trivial. 
We will suspend a minimal expansive homeomorphism
of a Cantor set under a specific return time function. 
The example also shows that
a positive kinematic expansive flow may not be negative expansive.

\begin{ejp}
Let $\theta\in \R$ be an irrational number and consider the rotation 
$R\colon[0,1)\to[0,1)$
given by $R(x)=x+\theta \mod 1$. By
\emph{splitting} along the orbit of 0 under $R$ we obtain a minimal
expansive homeomorphism $f\colon l\to l$ on a Cantor set $l$. Now
let $x_n=R^n (0)$ and choose an increasing sequence of positive
integers $n_j$ such that $x_{n_j}$ is strictly decreasing to 0.
Next find a sequence $\delta_j$ decreasing to 0 such that,
defining $I_j=[x_{n_j},x_{n_j}+\delta_j]$, $I_j\cap
I_k=\emptyset$ if $j\neq k$. 
Define a function $T\colon
[0,1)\to\R^+$ by the conditions: 
\begin{enumerate}
 \item $T(x_{n_j})=1+1/j$ and $T(x_{n_j}+\delta_j)=1$, 
 \item extend by linearity between the end points of each $I_j$ and 
 \item $T(x)=1$ otherwise. 
\end{enumerate}
Note that since the discontinuities of $T$
occurs at the points $x_{n_j}$, $T$ can be extended to a
continuous function on the Cantor set $l$.

Let $\phi$ be the suspension flow of $R\colon l\to l$ under the time function $T$. 
Given $x\in [0,1)$ in the orbit of 0 under $R$ denote by $x^+,x^-\in l$ the splitting points of $x$.
Note that
$\dist(f^k(0^-),f^k(0^+))\to 0$ as $k\to \pm\infty$ and that there
exists $\delta>0$ such that if 
$x\notin\{f^i(0^\pm):i\geq 0\}$
and $y\neq x$ then there is $k\geq 0$ such that $\dist
(f^k(x),f^k(y))>\delta$. If $T_n\colon l\to \R$ is given by
$\phi_{T_n(x)}(x)=f^n(x)$ for all $x\in l$ then we have
$$T_{n_j}(0^-)-T_{n_j}(0^+)=\sum_{k=1}^j 1/k.$$ 

By Proposition \ref{kexpsuspension} (the arguments in its proof) we conclude that
$\phi$ is positive kinematic expansive. Note that
$\dist(\phi_t(0^-),\phi_t(0^+))\to 0$ as $t\to-\infty$ and then
$\phi$ is not negative kinematic expansive.
\end{ejp}

\subsection{Kinematic bi-expansive flows}
\label{secBiExp}
In this brief section we wish to remark the non-existence of singularities for a 
flow being simultaneously positive and negative kinematic expansive. 
Let $\phi$ be a continuous flow on a compact metric space $X$ and
define the inverse flow $\phi^{-1}$ as $\phi^{-1}_t=\phi_{-t}$.
\begin{df}
 We say that $\phi$ is \emph{kinematic bi-expansive} if $\phi$ and $\phi^{-1}$ are positive kinematic expansive.
\end{df}

Examples of such flows are the periodic annulus (Example \ref{annulus}) and 
the horocycle flow of a negatively curved surface \cite{Gura}.

\begin{prop}
 If $\phi$ is a kinematic bi-expansive flow on a compact metric space $X$ 
 then every singularity is an isolated point of the space. 
 Therefore, if $X$ is connected with more than one point then there are no singularities.
\end{prop}

\begin{proof}
 Positive expansiveness implies that singular points are repellers 
 and negative expansiveness implies that they are attractors. 
 Then, singularities are isolated points of the space.
\end{proof}

\section{Robust kinematic expansiveness}
\label{secRobust}

In this section we study the persistence of expansiveness under perturbations of the velocity field in the $C^1$-topology. 
On surfaces there are no robust geometric expansive flows because small $C^1$-perturbations gives rise to periodic orbits (see \cite{Pe}) 
and this is an obstruction to geometric expansiveness (see \cite{Artigue}). 
As a corollary we have that there are no robust geometric expansive flows on 
three dimensional manifolds with non-empty boundary.

On surfaces we will consider robust kinematic expansiveness in the conservative framework. 
On manifolds of dimension greater than two we will prove that robust kinematic expansiveness is equivalent with geometric expansiveness.

\subsection{Positive expansiveness in the annulus}

Let $A\subset \R^2$ be the annulus bounded by two simple closed $C^1$ curves as in Figure \ref{figuraAnillo}.
Denote by $\X^1_\mu(A)$ the vector space of $C^1$ vector fields $X$ defined in $A$ such that 
\begin{enumerate}
 \item $\dive(X)=0$ and 
 \item $X$ is parallel to $\partial A$ in $\partial A$.
\end{enumerate}

\begin{figure}[htbp]
\begin{center}
\includegraphics{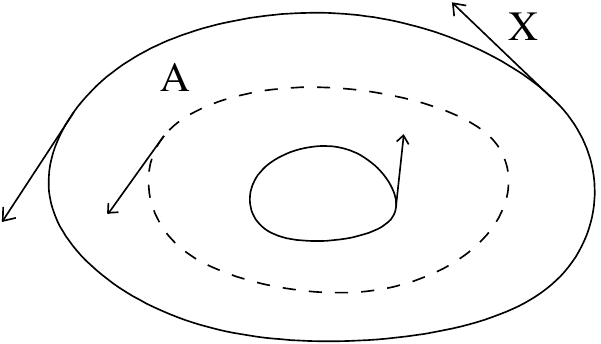}
    \caption{A vector field in the annulus tangent to the boundary.}
   \label{figuraAnillo}
\end{center}
\end{figure}
\begin{df}
 We say that $X\in \X^1_\mu(A)$ is \emph{robustly positive kinematic expansive} if there is a $C^1$-neighborhood of $X$ in 
 $X^1_\mu(A)$ such that every vector field in this neighborhood gives rise to a positive kinematic expansive flow.
\end{df}

If $X=(a,b)$ denote $X^\perp=(-b,a)$.

\begin{thm}
 Let $X\in\X^1_\mu(A)$ be a non-vanishing vector field and define $Z=X^\perp/\|X\|^2$. 
 If 
 \begin{equation}
  \label{eqRobexpAnillo}
  \dive(Z)\neq 0
 \end{equation}
 on every point of $A$ then $X$ is
 robustly positive kinematic expansive.
\end{thm}

\begin{proof}
 Let us first recall that if $\dive(X)=0$ then there are no wandering points and since 
 $X$ has no singularities, we have that every orbit is periodic because no other kind of recurrence is possible in the annulus 
 in our hypothesis. 
 It implies that the flow is a suspension of the identity in a global cross section. 
 Then, in order to prove kinematic expansiveness it is enough to prove that different periodic orbits have different periods. 
 
 Let $\gamma$ be a periodic orbit of $X$. 
 The period of $\gamma$, denoted by $T(\gamma)$, can be calculated as follows:
 \[
  T(\gamma)=\int_\gamma \frac{1}{\|X\|}\, d\gamma=\int_\gamma Z\cdot n\,d\gamma,
 \]
where $n$ is the normal vector of $\gamma$ in the direction of $Z$. 
That is, the period of $\gamma$ is the flow of $Z$ through $\gamma$. 
Now consider two periodic orbits $\gamma_1$ and $\gamma_2$ bounding a region $R$ as in Figure \ref{figRegionR}.
\begin{figure}[htbp]
\begin{center}
\includegraphics{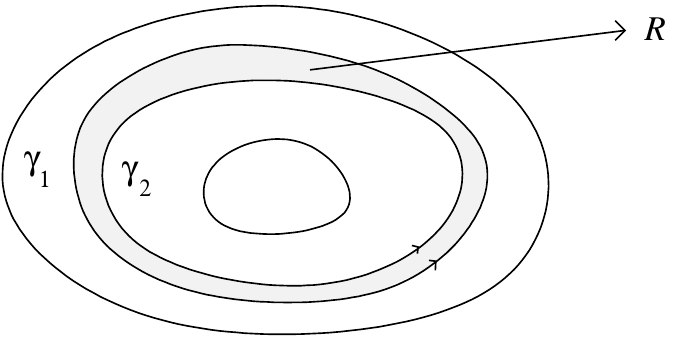}
    \caption{Region $R$ bounded by two periodic orbits.}
   \label{figRegionR}
\end{center}
\end{figure}

Applying Green's Theorem we have that 
\[
 \int_{\gamma_2} Z\cdot n\,d\gamma_2-\int_{\gamma_1} Z\cdot n\,d\gamma_1 = \iint_R \dive(Z)\, dx\, dy.
\]
And then
\[
 T(\gamma_2)-T(\gamma_1)=\iint_R \dive(Z) \, dx\, dy\neq 0.
\]
Then we have proved that different periodic orbits have different periods and then $X$ is kinematic expansive. 
It only rests to notice that condition (\ref{eqRobexpAnillo}) cannot be lost by a $C^1$ small perturbation of $X$.
\end{proof}

\begin{ejp}
 Given $r_1>0$ and $r_2>r_1$ consider the annulus $A\subset\R^2$ given by $r_1^2\leq x^2+y^2\leq r_2^2$. 
 Given a smooth non-vanishing function $f\colon\R\to\R$ define $X_f(x,y)=f(r^2)(y,-x)$, where $r^2=x^2+y^2$. 
In this case $$Z=\frac1{r^2f(r^2)}(x,y)$$ and 
\[
 \dive(Z)=-\frac{2f'}{f^2}.
\]
Therefore, $X_f$ is robust kinematic expansive in $A$ 
if $f'\neq 0$ in $[r_1,r_2]$.
\end{ejp}

\subsection{Robust expansiveness on manifolds}

Let $X$ be a $C^1$ vector field of a closed manifold $M$ of dimension $n\geq 3$. 
Assume that $M$ is endowed with a smooth structure and a smooth Riemannian metric. 
In this section we also assume that $X$ has no singularities.

\begin{df}
 We say that $X$ is $C^1$-\emph{robust kinematic (or geometric) expansive} if every 
 vector field in a suitable $C^1$-neighborhood of $X$ is kinematic (or geometric) expansive.
\end{df}

\begin{thm}
 Every $C^1$-robust kinematic expansive vector field without singularities on a closed smooth manifold is geometric expansive.
\end{thm}

\begin{proof}
Consider $X$ a $C^1$-robust kinematic expansive vector field.
Let us start proving that periodic orbits of $X$ are hyperbolic.
In Proposition 1 of \cite{MSS} it is proved (with standard perturbation techniques) 
that if a periodic orbit is not hyperbolic then there is a 
$C^1$-close vector field $Y$ with an invariant annulus $A$
filled with periodic orbits of $Y$. 
This gives a contradiction with geometric expansiveness, but in our case we have to give more arguments. 
Consider a new perturbation $Z$ such that $A$ is $Z$-invariant but with at least one non-periodic orbit.
This easily contradicts Proposition \ref{suspKexpIC1}.

Therefore, we have proved that 
every periodic orbit of every vector 
field in a suitable neighborhood of $X$ is hyperbolic. 
A vector field with this property is usually called as \emph{star flow}. 
In \cite{GanWen} it is proved (see Theorem A) that 
non-singular star flows satisfy Axiom A, i.e., periodic orbits are 
dense in $\Omega(X)$ and $\Omega(X)$ is hyperbolic.

Now we prove the quasi-transversality condition, that is: 
\begin{equation}
 \label{qtc}
 T_x W^s(x)\cap T_x W^u(x)=\{0_x\}
\end{equation}
for all $x\in M$, where $W^s(x)$ is the stable manifold and $W^u(x)$ is 
the unstable manifold of $x$ defined as usual.
For $x\in\Omega(X)$ the quasi-transversality condition holds because $\Omega(X)$ is hyperbolic. 
Consider $x\notin \Omega(X)$ and, arguing by contradiction, assume that (\ref{qtc}) does not hold.
With a 
$C^1$-perturbation $Y$ of $X$ 
we can also assume that $x$ is in the stable set of a periodic orbit $\gamma_1$ and also 
in the unstable set of a periodic orbit $\gamma_2$. 
With another perturbation $Z$ we can suppose that the intersection of the 
stable manifold of $\gamma_1$ with the unstable manifold of $\gamma_2$ and a local cross section through $x$ contains 
an arc $l$ containing $x$. 
Now it is easy to arrive to a contradiction using the arguments in the proof of Proposition \ref{suspKexpIC1}. 

Since $X$ satisfies Axiom A and the quasi-transversality condition, we can apply the results of \cite{MSS} to conclude that 
$X$ is in fact robust geometric expansive.
\end{proof}

\end{document}